\documentclass[a4paper,11pt]{amsart}

\usepackage{geometry}
\usepackage{amssymb,mathrsfs,amsmath,amsthm,amsbsy}
\usepackage{enumitem}
\usepackage{autobreak}
\allowdisplaybreaks
\usepackage{bm}
\usepackage{amsmath}
\usepackage{tabularx}
\usepackage{tabulary}

\usepackage{graphicx}
\usepackage{color}
\usepackage{booktabs}
\usepackage{hyperref}
\usepackage{cleveref}

\newtheorem{theorem}{Theorem}[section]
\newtheorem{lemma}{Lemma}[section]
\newtheorem{proposition}{Proposition}[section]

\newtheorem{definition}{Definition}[section]
\newtheorem{remark}{Remark}[section]
\newtheorem{example}{Example}[section]

\numberwithin{equation}{section}
\numberwithin{figure}{section}

\begin{document}

\title{Sharp conditions for fully nonlinear equations with $(n-1)$-Hessian matrix: Part I. Monge--Amp\`ere case}

\author{Jingwen Ji}
\address{School of Mathematics and Shing-Tung Yau Center, Southeast University, Nanjing 211189, P. R. China}
\email{jingwen\_ji@seu.edu.cn}

\author{Feida Jiang*}
\address{School of Mathematics and Shing-Tung Yau Center, Southeast University, Nanjing 211189, P. R. China; Shanghai Institute for Mathematics and Interdisciplinary Sciences, Shanghai 200433, P. R. China}
\email{jiangfeida@seu.edu.cn}
\thanks{Corresponding author: Feida Jiang.}

\author{Mengni Li}
\address{School of Mathematics, Southeast University, Nanjing 211189, P. R. China}
\email{lmn@seu.edu.cn}

\subjclass[2020]{Primary 35J60; Secondary 35B08, 35B44, 35J70, 35J75}
\keywords{Keller--Osserman condition; real $(n-1)$ Monge--Amp\`ere equation; $(n-1)$-convex entire subsolution.}

\begin{abstract}
We study positive $(n-1)$-convex entire subsolutions of the normalized real $(n-1)$ Monge--Amp\`ere equation $\operatorname{det}^{\frac1n}\bigl(\Delta uI-D^2u\bigr)=f(u)+g(u)|Du|^q$ in $\mathbb R^n$. Under two sets of structural assumptions on the continuous, non-decreasing coefficients $f$ and $g$, we establish necessary and sufficient conditions for the existence of entire subsolutions when $q<2-\frac1n$. Our results cover the pure-gradient degenerate case $f\equiv0$ with $0<q<1$, the gradient-singular case $q<0$, and the case where $g$ changes sign. 
For $q=2-\frac1n$, we recover the Keller--Osserman condition for the gradient-independent equation, prove nonexistence when $g$ is eventually positive, and obtain an exponentially weighted condition when $g$ has a nonpositive limit at infinity. This is the first of two companion papers and treats the determinant case.
\end{abstract}

\maketitle

\setcounter{tocdepth}{1}
\tableofcontents

\section{Introduction}\label{Section1}

We study $(n-1)$-convex entire subsolutions of the normalized real $(n-1)$ Monge--Amp\`ere equation with an additive gradient-dependent inhomogeneous term
\begin{equation}\label{eq1.1}
\operatorname{det}^{\frac{1}n}\bigl(\Delta uI-D^2u\bigr)=f(u)+g(u)|Du|^q,\quad x\in\mathbb R^n,
\end{equation}
where $n\geq2$, $q\in\mathbb R$, and $Du$, $D^2u$, and $\Delta u$ denote the gradient, Hessian matrix, and Laplacian of $u$, respectively. Here normalized refers to taking the $n$th root, which makes the operator positively homogeneous of degree one.
The functions $f$ and $g$ are continuous and non-decreasing. If $\lambda_1,\ldots,\lambda_n$ are the eigenvalues of $D^2u$, then the eigenvalues of $\Delta uI-D^2u$ are $\sum_{j\ne i}\lambda_j$, $i=1,\ldots,n$. Thus the operator is determined by the $(n-1)$-fold sums of the Hessian eigenvalues. In particular, when $n=2$, $\det \bigl(\Delta uI-D^2u\bigr)=\det D^2u$, and the operator in \eqref{eq1.1} is the normalized classical Monge--Amp\`ere operator.

To work in the elliptic setting, a function $u\in C^2(\Omega)$ is called \emph{$(n-1)$-convex} in a domain $\Omega\subseteq\mathbb R^n$ if
$$
\Delta uI-D^2u\geq0\quad\text{in }\Omega,
$$
and it is called \emph{strictly $(n-1)$-convex} if $\Delta uI-D^2u$ is positive definite. We denote the class of $(n-1)$-convex functions in $\Omega$ by $\Phi^n(\Omega)$.
The operator $\operatorname{det}^{\frac{1}n}\bigl(\Delta uI-D^2u\bigr)$ is degenerate elliptic with respect to $(n-1)$-convex functions.

Our objective is to characterize the existence of positive nontrivial $(n-1)$-convex entire subsolutions by the growth of $f$ and $g$, extending the gradient-independent problem studied in \cite{JJL2025}. The additive gradient term introduces three issues. For $f\equiv0$ and $0<q<1$, constant solutions coexist with nontrivial branches, so we construct an immediately increasing radial solution. For $q<0$, the equation is singular at zero-gradient points, and we use the strong solution class specified below. When $g$ changes sign, positivity of the inhomogeneous term along the constructed radial solution must be proved rather than assumed. The determinant structure provides the radial divergence identity \eqref{Volterra_differential}, which is the starting point for both the local construction and the global derivative estimates.

\subsection{Background and Motivation}

Complex analogues of the $(n-1)$ Monge--Amp\`ere operator arise in extensions of Yau's theorem \cite{Yau1978} to Hermitian geometry.
This operator is closely related to the $(n-1)$-plurisubharmonic functions developed by Harvey and Lawson \cite{HL2011,HL2012} and the form-type Calabi--Yau equations introduced by Fu, Wang, and Wu \cite{FWW2010,FWW2015} in the study of balanced metrics.
Tosatti and Weinkove \cite{TW2017,TW2019} studied the associated equations on compact K\"ahler and Hermitian manifolds; related almost complex and almost Hermitian equations were treated in \cite{CTW2019,CHZ2023}. The problem of prescribing volume forms of Gauduchon metrics \cite{G1984} was resolved by Sz\'ekelyhidi, Tosatti and Weinkove \cite{STW2017}. Fully nonlinear Hermitian equations with gradient terms were also studied by Guan and Nie \cite{GN2023}.

In the real setting, $(n-1)$-convexity corresponds to the case $p=n-1$ of the $p$-plurisubharmonicity studied by Harvey and Lawson \cite{HL2013}.
On Riemannian manifolds, Jiao and Liu \cite{JL2021} established second-order estimates and existence results for Hessian type equations involving $\sigma_k(\Delta u I - D^2 u)$, which includes the real $(n-1)$ Monge--Amp\`ere equation as a special case. Related curvature, Hessian quotient, and sum-Hessian equations were studied in \cite{CJ2022,CTX2021,CTX2023,CDH2023,GLT2025,RW2025,RW2026}. For the real $(n-1)$ Monge--Amp\`ere equation, the literature also includes the Neumann problem \cite{Deng2020}, entire solvability \cite{JJL2025}, boundary blow-up problem \cite{JDJ2025,JJ2026}, and viscosity solutions in exterior domains \cite{Dai2026}.

The growth question is rooted in the work of Keller \cite{K1957} and Osserman \cite{O1957}. For a positive, continuous, non-decreasing function $f$, the semilinear differential inequality $\Delta u\geq f(u)$ admits an entire solution if and only if
\begin{equation}\label{classical_KO}
\int^{+\infty}\left(\int_0^\tau f(t)\,dt\right)^{-\frac{1}{2}}d\tau=+\infty,
\end{equation} which is known as the Keller--Osserman condition.
Here and throughout the paper, an omitted lower limit in an improper integral may be replaced by any positive constant.
The Keller--Osserman theory has been extended to several nonlinear elliptic settings. Naito and Usami \cite{NU1997} studied the existence and nonexistence of positive entire solutions of quasilinear elliptic inequalities of the form
$$\operatorname{div}\bigl(A(\vert{}Du\vert{})Du\bigr)\geq f(u).$$
Filippucci, Pucci and Rigoli \cite{FPR2010} obtained existence and nonexistence results for nonnegative entire weak solutions of weighted $p$-Laplacian inequalities with gradient terms, using generalized Keller--Osserman conditions.
On weighted Riemannian manifolds, Mari, Rigoli and Setti \cite{MRS2010} established generalized Keller--Osserman conditions for inequalities of the forms $$\mathcal Lu\geq b(x)f(u)\ell(|\nabla u|)$$ and $$\mathcal Lu\geq b(x)f(u)\ell(|\nabla u|)-g(u)h(|\nabla u|),$$
where $\mathcal L$ includes the $p$-Laplacian and the mean curvature operator.
For fully nonlinear equations, Ji and Bao \cite{JB2010} established generalized Keller--Osserman conditions for standard $k$-Hessian inequalities
$$
S_k^{\frac{1}k}(D^2u)\geq f(u).
$$
Bao, Ji and Li \cite{BJL2012} established existence and nonexistence results for positive entire subsolutions of fully nonlinear equations related to the $k$-Yamabe problem. Bao and Feng \cite{BF2022} studied $p$-$k$-Hessian inequalities of the form
$$S_k[D(\vert{}Du\vert{}^{p-2}Du)] \geq f(u).$$
Hessian type equations with gradient-dependent inhomogeneous terms were also investigated in \cite{JJL2023}.
In the viscosity setting, Capuzzo Dolcetta, Leoni and Vitolo \cite{CDLV2014,C2016} studied entire subsolutions of fully nonlinear degenerate elliptic inequalities
$$F(x,D^2u)\geq f(u)+g(u)|Du|^q.$$
The model operators for $F$ include the degenerate maximal Pucci operator $\mathcal M_{0,1}^+(X)= \sum_{\mu_i(X)>0}\mu_i(X)$ and the $k$-partial trace operator $\mathcal P_k^+(X)= \mu_{n-k+1}(X)+\ldots+\mu_n(X)$, where $\mu_1(X)\leq\ldots\leq\mu_n(X)$ are the eigenvalues of $X$.
Bhattacharya and Mohammed \cite{BM2021} studied standard $k$-Hessian equations with zero- and first-order terms on unbounded domains and obtained Phragm\'en--Lindel\"of and Liouville type results under Keller--Osserman type assumptions.
Further extensions of Keller--Osserman type conditions for entire subsolutions of Monge--Amp\`ere and Hessian type equations can be found in \cite{Dai2020,DL2020,YZ2023,ZY2023}.
For the real $(n-1)$ Monge--Amp\`ere operator, our earlier work \cite{JJL2025} established sharp entire solvability conditions for the gradient-independent equation
$$
\operatorname{det}^{\frac{1}n}\bigl(\Delta uI-D^2u\bigr)=b(x)f(u),
$$ and conditions on $b$ for the existence of entire large or bounded solutions.

Classical Keller--Osserman theory relates the existence of entire subsolutions to the growth of the zero-order nonlinearity. Motivated by the preceding results, we investigate generalized Keller--Osserman conditions for the real $(n-1)$ Monge--Amp\`ere equation \eqref{eq1.1}, taking into account the additive interaction between $f(u)$ and $g(u)|Du|^q$. This interaction cannot be described by the growth of $f$ alone. Indeed, a positive gradient coefficient may force finite-radius blow-up, and a negative coefficient may compensate for a rapidly increasing zero-order term. We consider two structural conditions, \textup{(A)} and \textup{(B)}, stated below. Under condition \textup{(A)}, the coefficients are nonnegative, and we treat $q\leq2-\frac1n$, including the singular case $q<0$ and the pure-gradient case $f\equiv0$, $0<q<1$. Under condition \textup{(B)}, $g$ may change sign, $f(t)>0$ for $t>0$, and we consider $0<q\leq2-\frac1n$, with the resulting existence conditions depending on the asymptotic sign of $g$. We establish existence and nonexistence results for classical entire subsolutions in the nonsingular cases and strong entire subsolutions in the singular case.

\subsection{Main Results}\label{SectionMainResults}

Before stating the main theorems, we specify the notions of entire subsolutions used below.

\begin{definition}[Admissible entire subsolutions]\label{def:entire_(n-1)-convex_subsolutions}
An $(n-1)$-convex entire subsolution of \eqref{eq1.1} is understood in one of the following senses.
\begin{enumerate}[label=(\roman*)]
\item A function $u\in C^2(\mathbb R^n)$ is called an {\bfseries\boldmath$(n-1)$-convex classical entire subsolution} of \eqref{eq1.1} if $\Delta uI-D^2u\geq0$ in $\mathbb R^n$ and
$$
\operatorname{det}^{\frac{1}n}(\Delta uI-D^2u)\geq f(u)+g(u)|Du|^q,\quad x\in\mathbb R^n.
$$

\item Let $q<0$ and $\delta\geq1$. A function $u\in C^1(\mathbb R^n)\cap C^2(\mathbb R^n\setminus\{0\})\cap W_{\mathrm{loc}}^{2,\delta}(\mathbb R^n)$ is called an {\bfseries\boldmath$(n-1)$-convex strong entire subsolution} with an isolated singularity at the origin if $\Delta uI-D^2u\geq0$ in $\mathbb R^n\setminus\{0\}$, $|Du|>0$ a.e. in $\mathbb R^n\setminus\{0\}$, and
$$
\operatorname{det}^{\frac{1}n}(\Delta uI-D^2u)\geq f(u)+g(u)|Du|^q\quad\text{for a.e. }x\in\mathbb R^n\setminus\{0\}.
$$
\end{enumerate}
\end{definition}

At $q=0$, we use the convention $|Du|^0:=1$, including at $Du=0$. By Definition \ref{def:entire_(n-1)-convex_subsolutions}, an $(n-1)$-convex strong entire subsolution belongs to $C^1(\mathbb R^n)\cap C^2(\mathbb R^n\setminus\{0\})$. We do not repeat these requirements in the main theorems and state only the relevant Sobolev regularity and exponent range explicitly.

For $\delta\geq1$, we use the conventions
$$
W_{\mathrm{loc}}^{2,\delta}[0,+\infty):=\left\{v:[0,+\infty)\to\mathbb R:v\in W^{2,\delta}(0,R)\ \text{for any }R>0\right\}
$$
and
$$
W_{\mathrm{loc}}^{2,\delta}(\mathbb R^n):=\left\{u:\mathbb R^n\to\mathbb R:u\in W^{2,\delta}(B_R)\ \text{for any }R>0\right\}.
$$

We now present the main theorems, which give sharp Keller--Osserman conditions for the existence of $(n-1)$-convex entire subsolutions in $\mathbb R^n$.

We first impose the following nonnegative structural condition:
\begin{enumerate}
\item[\textbf{(A)}] {\bfseries\boldmath $f,g\in C(\mathbb R,[0,+\infty))$ are non-decreasing and satisfy $f(t)+g(t)>0$ for all $t>0$.}
\end{enumerate}
Condition (A) includes the gradient-independent equation when $g\equiv0$ or $q=0$, as well as the pure-gradient equation when $f\equiv0$ and $q\neq0$. If $f\equiv0$ and $q>0$, every positive constant is a solution. Our degenerate alternative is restricted to $0<q<1$, and its existence proof selects an immediately increasing nontrivial radial branch. For $q<0$, we use the common Sobolev class specified below.

Our first main theorem gives a sharp Keller--Osserman characterization under condition (A).

\begin{theorem}[Entire subsolutions under condition (A)]\label{Th1.1}
Assume that $f$ and $g$ satisfy condition (A) and let $q\in\mathbb R$. If $q>0$, assume additionally that either $f(t)>0$ for all $t>0$, or that $0<q<1$ and $f\equiv0$.
\begin{enumerate}[label=(\roman*)]
\item Let $0\leq q<2-\frac{1}n$. Equation \eqref{eq1.1} admits a positive nontrivial $(n-1)$-convex classical entire subsolution if and only if
\begin{equation}\label{KO_n}
\int^{+\infty}\frac{dt}{\left(\int_0^t f^{\frac{n}{n-1}}(s)ds\right)^{\frac{n-1}{2n-1}}+\left(\int_0^t g^{\frac{n}{n-1}}(s)ds\right)^{\frac{n-1}{2n-1-qn}}}=+\infty.\tag{$KO_n$}
\end{equation}
\stepcounter{equation}

\item Let $q<0$. Equation \eqref{eq1.1} admits a positive nontrivial $(n-1)$-convex strong entire subsolution
$u\in W_{\mathrm{loc}}^{2,\delta}(\mathbb R^n)$ for $\delta\in [1,\frac{n(1-q)}{1-nq})$
if and only if \eqref{KO_n} holds.

\item Let $q=2-\frac{1}n$. If $g\equiv0$ on $(0,+\infty)$, then equation \eqref{eq1.1} admits a positive nontrivial $(n-1)$-convex classical entire subsolution if and only if
\begin{equation}\label{KO_n_critical_zero}
\int^{+\infty}\frac{dt}{\left(\int_0^t f^{\frac{n}{n-1}}(s)ds\right)^{\frac{n-1}{2n-1}}}=+\infty.\tag{$KO_n^{(0)}$}
\end{equation}\stepcounter{equation} If $g\not\equiv0$ on $(0,+\infty)$, then equation \eqref{eq1.1} admits no positive nontrivial $(n-1)$-convex entire subsolution.
\end{enumerate}
\end{theorem}

Throughout the paper, nontrivial means nonconstant. The nonconstancy requirement in Theorem \ref{Th1.1}~(i) is essential in the pure-gradient case $0<q<1$ and $f\equiv0$, since every positive constant is then a solution independently of \eqref{KO_n}. In this case, \eqref{KO_n} characterizes the existence of an entire subsolution, equivalently the existence of an increasing convex global radial solution constructed in Section \ref{Section3}.

When $g\equiv0$ on $(0,+\infty)$, \eqref{KO_n} reduces to \eqref{KO_n_critical_zero}, which is the Keller–Osserman condition for the gradient‑independent equation obtained in \cite{JJL2025}.
When $g\not\equiv0$ on $(0,+\infty)$, Theorem \ref{Th1.1} excludes positive nontrivial $(n-1)$-convex entire subsolutions for $1<q\leq2-\frac{1}n$ under its stated hypotheses; see Remark \ref{Rem4.critical_exponents}.

We next allow $g$ to change sign. For $q>0$, consider the following alternative condition:
\begin{enumerate}
\item[\textbf{(B)}] {\bfseries\boldmath $f\in C(\mathbb R,[0,+\infty))$ and $g\in C(\mathbb R,\mathbb R)$ are non-decreasing and satisfy $f(t)>0$ for all $t>0$.}
\end{enumerate}
Set $g^+(t):=\max\{g(t),0\}$ and $g^-(t):=\max\{-g(t),0\}$.
For $0<q\leq 2-\frac{1}n$, define
\begin{equation}\label{Pn}
P_n(t):=\frac{(g^-(t))^{\frac{2n-1}{q(n-1)}}}{(f(t))^{\frac{2n-1-qn}{q(n-1)}}},\quad t>0.
\end{equation}
The function $P_n$ is well defined on $(0,+\infty)$ under condition (B).

Under condition (B), the behavior of the gradient term is determined by the asymptotic sign of $g$. By monotonicity, $g$ admits an extended‑real limit at infinity. If this limit is positive, the growth is governed by the positive part $g^+$; if it is nonpositive, the modified Keller--Osserman condition involves the weight $P_n$ generated by $g^-$. We now state our second main result.

\begin{theorem}[Entire subsolutions under condition (B)]\label{Th1.2}
Assume that $f$ and $g$ satisfy condition (B), and $q>0$.
\begin{enumerate}[label=(\roman*)]
\item Let $0<q<2-\frac{1}n$ and $\lim\limits_{t\to+\infty}g(t)>0$. Equation \eqref{eq1.1} admits a positive nontrivial $(n-1)$-convex classical entire subsolution if and only if
\begin{equation}\label{KO_n_2}
\int^{+\infty}\frac{dt}{\left(\int_0^t f^{\frac{n}{n-1}}(s)ds\right)^{\frac{n-1}{2n-1}}+\left(\int_0^t(g^+(s))^{\frac{n}{n-1}}ds\right)^{\frac{n-1}{2n-1-qn}}}=+\infty.\tag{$KO_n^{(g^+)}$}
\end{equation}
\stepcounter{equation}

\item Let $0<q\leq 2-\frac{1}n$ and $\lim\limits_{t\to+\infty}g(t)\leq0$. Equation \eqref{eq1.1} admits a positive nontrivial $(n-1)$-convex classical entire subsolution if and only if
\begin{equation}\label{KO_n_3}
\int^{+\infty}\frac{dt}{\left(\int_0^t f^{\frac{n}{n-1}}(s)e^{-\int_s^tP_n(\tau)d\tau}ds\right)^{\frac{n-1}{2n-1}}}=+\infty.\tag{$KO_n^{(g^-)}$}
\end{equation}
\stepcounter{equation}

\item Let $q=2-\frac{1}n$ and $\lim\limits_{t\to+\infty}g(t)>0$. Then equation \eqref{eq1.1} admits no positive nontrivial $(n-1)$-convex classical entire subsolution.
\end{enumerate}
\end{theorem}

\begin{remark}\label{Rem1.KO-lower-limit}
The omitted lower limits in the outer improper integrals \eqref{KO_n}, \eqref{KO_n_critical_zero}, \eqref{KO_n_2}, and \eqref{KO_n_3} may be replaced by arbitrary positive constants without affecting convergence or divergence at infinity.
Under condition (B), the coefficient $1$ in the exponential weight of \eqref{KO_n_3} is a normalization. By Lemma \ref{Lem2.4}, replacing $e^{-\int_s^tP_n(\tau)d\tau}$ by $e^{-\lambda\int_s^tP_n(\tau)d\tau}$ for any $\lambda>0$ leaves the divergence condition unchanged.
\end{remark}

\begin{remark}\label{Rem1.positivity}
If $f(t)>0$ for all $t\in\mathbb R$, the positivity restriction on the subsolution may be removed. Under the modified conditions $(A)'$ and $(B)'$ in Remark \ref{Rem5.3}, the same Keller--Osserman conditions characterize nontrivial $(n-1)$-convex entire subsolutions taking arbitrary real values, in the corresponding ranges of $q$.
\end{remark}

\subsection{Main Features and Organization}

The main features of this paper are summarized as follows.

\begin{enumerate}[label=(\roman*)]
\item We establish necessary and sufficient conditions for positive nontrivial $(n-1)$-convex entire subsolutions when $q<2-\frac{1}n$. These include the pure-gradient degenerate case $f\equiv0$ with $0<q<1$, strong entire subsolutions in the gradient-singular case $q<0$ under condition (A), and classical entire subsolutions for sign-changing gradient coefficients under condition (B).

\item The divergence identity \eqref{Volterra_differential} yields an integral formulation of the radial problem involving both the unknown function and its derivative.
Adapting the Euler method in \cite{JB2010,JJL2025,JTX2026}, we obtain uniform convergence of the approximations and $L^\infty$ convergence of their piecewise derivatives. In the singular case, positive lower slope bounds ensure that the limit solves the original equation with its negative gradient power. In the pure-gradient degenerate case, a divergence-form transformation allows us to apply \cite[Theorem 1.1]{JTX2026} to construct an immediately increasing solution.

\item The dimension-dependent exponent $q=2-\frac{1}n$ is the structural threshold at which the radial estimates change from power-type to exponential form. At this value, an eventually positive gradient coefficient prevents entire subsolutions, and the gradient-independent case $g\equiv0$ and the case $\lim\limits_{t\to+\infty}g(t)\leq0$ give distinct Keller--Osserman conditions. Under eventual positivity of $g$, nonexistence holds for $1<q<2-\frac{1}n$, as explained in Remark \ref{Rem4.critical_exponents}. Our conditions reduce to those in \cite{JJL2025} when $g\equiv0$ or $q=0$.

\item We establish an equivalence between entire subsolution existence and global radial solvability. For $q<0$ or $0<q<1$, direct gradient linearization is unavailable at zero-gradient contact points. By translation, the comparison is transferred to points where both functions are $C^2$ and the gradients nonzero. Local linearization in the strictly $(n-1)$-convex cone, the strong maximum principle, and an open-and-closed argument then yield the required comparison with radial blow-up barriers. This argument uses continuity and non-decreasing property of $f$ and $g$, without requiring differentiability or strict monotonicity.
\end{enumerate}

This is the first of two companion papers concerning
$$S_k^{\frac{1}k}[\Delta uI-D^2u]=f(u)+g(u)|Du|^q.$$
For the determinant case $k=n$,  the critical gradient exponent is $2-\frac{1}n$, and a radial divergence identity provides an explicit integral representation for the derivative of radial solutions.
For $2\leq k<n$, the radial equation is implicit in the second derivative, and the explicit divergence formula used here does not directly apply. The case $k=1$ reduces to a semilinear equation, since $S_1[\Delta uI-D^2u]=(n-1)\Delta u$. 
For the $k$-Hessian case $1 \leq k < n$, the radial equation has a different structure and leads to different Keller--Osserman conditions. This issue will be studied in the second paper and $k$-independent Keller-Osserman conditions will be obtained.

\subsection{Organization}
The remainder of the paper is organized as follows.
Section \ref{Section2} establishes the qualitative and origin properties of radial solutions and records a general weighted-integral comparison.
Section \ref{Section3} constructs local increasing convex solutions.
Section \ref{Section4} derives two-sided derivative estimates and characterizes global radial solvability. Section \ref{Section5} proves the local comparison principle, the equivalence with entire subsolutions, and the main theorems.
Section \ref{Section6} presents examples and further remarks.


\section{Preliminaries on Radial Solutions}\label{Section2}

This section derives the radial form of \eqref{eq1.1} and establishes the regularity, asymptotic behavior, and qualitative properties of its radial solutions. Throughout the paper, we write $r=|x|$ for $x\in\mathbb R^n$.

We begin with the standard eigenvalue computation for radial functions.
\begin{lemma}[Lemma 2.1 of \cite{JJL2025}]\label{Lem2.1}
Suppose that $\varphi \in C^2[0,+\infty)$ and $\varphi^\prime(0)=0$. Then for $u(x)=\varphi(r)$, we have $u(x)\in C^2(\mathbb{R}^n)$ and the eigenvalues of $\Delta u I - D^2u$ are given by
\begin{equation}\label{2.1}
\lambda(\Delta uI-D^2u)
=\left\{
\begin{aligned}
&\left((n-1)\dfrac{\varphi^\prime}{r},\varphi^{\prime\prime}+(n-2)\dfrac{\varphi^\prime}{r},\ldots,\varphi^{\prime\prime}+(n-2)\dfrac{\varphi^\prime}{r}\right),&&r>0,\\ &\left((n-1)\varphi^{\prime\prime}(0),\ldots,(n-1)\varphi^{\prime\prime}(0)\right),&&r=0.
\end{aligned}
\right.
\end{equation}
Consequently,
\begin{equation}\label{2.2}
\det(\Delta uI-D^2u)=
\left\{
\begin{aligned}
&(n-1)\dfrac{\varphi^\prime(r)}r\left[\varphi^{\prime\prime}(r)+(n-2)\dfrac{\varphi^\prime(r)}r\right]^{n-1},&&r>0,\\ &\bigl((n-1)\varphi^{\prime\prime}(0)\bigr)^n,&&r=0.
\end{aligned}
\right.
\end{equation}
\end{lemma}

Let
$$
h(r):=f(\varphi(r))+g(\varphi(r))|\varphi^\prime(r)|^q.
$$
When $q<0$, the term $|\varphi^\prime(r)|^q$ is defined only where $\varphi^\prime(r)\neq0$. By \eqref{2.1}, an $(n-1)$-convex radial solution $u(x)=\varphi(|x|)$ satisfies
\begin{equation}\label{radial_admissibility}
\varphi^\prime(r)\geq0,\quad \varphi^{\prime\prime}(r)+(n-2)\frac{\varphi^\prime(r)}r\geq0,\quad r>0.
\end{equation}
It follows from \eqref{2.2} that an $(n-1)$-convex radial solution of \eqref{eq1.1} satisfies
\begin{equation}\label{sODE}
\left[\varphi^{\prime\prime}(r)+(n-2)\frac{\varphi^\prime(r)}r\right]^{\frac{n-1}{n}}\left[(n-1)\frac{\varphi^\prime(r)}r\right]^{\frac{1}n}=h(r),\quad r>0.
\end{equation}
In particular, $h\geq0$. A direct calculation rewrites \eqref{sODE} in divergence form
\begin{equation}\label{Volterra_differential}
\frac{d}{dr}\left(\left(r^{n-2}\varphi^\prime(r)\right)^{\frac{n}{n-1}}\right)=n(n-1)^{-\frac{n}{n-1}}r^{n-1}h^{\frac{n}{n-1}}(r).
\end{equation}
If $\varphi^\prime(0)=0$, then $r^{n-2}\varphi^\prime(r)\to0$ as $r\to0^+$. Integrating \eqref{Volterra_differential} over $[0,r]$ yields
\begin{equation}\label{Volterra}
\varphi^\prime(r)=\frac{n^{\frac{n-1}{n}}}{n-1}r^{2-n}\left(\int_0^r s^{n-1}h^{\frac{n}{n-1}}(s)ds\right)^{\frac{n-1}{n}},\quad r>0.
\end{equation}

When $g$ may change sign, the condition $h\geq0$ is not available before the radial solution has been constructed.
We set $h^+(r):=\max\{h(r),0\}$ and consider the auxiliary integral equation
\begin{equation}\label{mVolterra}
\varphi^\prime(r)=C_nr^{2-n}\left(\int_0^r s^{n-1}(h^+(s))^{\frac{n}{n-1}}ds\right)^{\frac{n-1}{n}},\quad r>0,
\end{equation}where $C_n=\frac{n^{\frac{n-1}{n}}}{n-1}$.
Equation \eqref{mVolterra} gives $\varphi^\prime(r)\geq0$ for $r>0$. At every point where $\varphi^\prime(r)>0$, differentiating \eqref{mVolterra} gives
\begin{equation}\label{second_positive_part}
\varphi^{\prime\prime}(r)+(n-2)\frac{\varphi^\prime(r)}r=(h^+(r))^{\frac{n}{n-1}}\left((n-1)\frac{\varphi^\prime(r)}r\right)^{-\frac{1}{n-1}}\geq0,
\end{equation}
and hence
\begin{equation}\label{modified_sODE}
\left[\varphi^{\prime\prime}(r)+(n-2)\frac{\varphi^\prime(r)}r\right]^{\frac{n-1}{n}}\left[(n-1)\frac{\varphi^\prime(r)}r\right]^{\frac{1}n}=h^+(r).
\end{equation}
Once $h\geq0$ is established, \eqref{mVolterra} and \eqref{modified_sODE} reduce to \eqref{Volterra} and \eqref{sODE}, respectively. This will follow from condition (A) for increasing solutions, or from Lemma \ref{Lem2.3}~(iii) under condition (B).

The divergence structure gives an explicit representation of $\varphi^\prime(r)/r$, which yields $C^2$ regularity at the origin in the nondegenerate cases. In the degenerate pure-gradient case $0<q<1$ with $f\equiv0$, the same structure yields $C^2$ regularity at the origin for immediately increasing radial solutions. For $q<0$, we obtain a common Sobolev estimate near the origin.

\begin{lemma}\label{Lem2.2}
Assume that $f,g\in C(\mathbb R)$ and $q\in\mathbb R$. Let $R>0$ and let $\varphi\in C^1[0,R]\cap C^2(0,R]$ satisfy \eqref{mVolterra}, $\varphi(0)=a>0$, and $\varphi^\prime(0)=0$. Then the following assertions hold.
\begin{enumerate}[label=(\roman*)]
\item If $q>0$ and $f(a)>0$, then $\varphi\in C^2[0,R]$ and $\varphi^{\prime\prime}(0)=\frac{f(a)}{n-1}$.

\item If $q=0$ and $f(a)+g(a)>0$, then $\varphi\in C^2[0,R]$ and $\varphi^{\prime\prime}(0)=\frac{f(a)+g(a)}{n-1}$.

\item Let $0<q<1$, $f\equiv0$, and $g(a)>0$. Suppose that $\varphi^\prime(r)>0$ on $(0,R_0]$ for some $R_0\in(0,R]$. Then $\varphi\in C^2[0,R]$ and $\varphi^{\prime\prime}(0)=0$.

\item Let $q<0$. Suppose that $\varphi^\prime(r)>0$ and $\varphi^{\prime\prime}(r)\geq0$ on $(0,R_0]$ for some $R_0\in(0,R]$. Then $\varphi\in W_{\mathrm{loc}}^{2,\delta}[0,R]$ for $\delta\in [1,\frac{n(1-q)}{1-nq})$.
\end{enumerate}
\end{lemma}

\begin{proof}
We first prove (i)--(ii). In each case, $\lim\limits_{r\to0^+}h(r)=h_0>0$, where $h_0=f(a)$ and $h_0=f(a)+g(a)$, respectively. Thus $h^+=h$ near the origin, and \eqref{mVolterra} reduces to \eqref{Volterra}. Since
$$
\lim\limits_{r\to0^+}\frac{1}{r^n}\int_0^r s^{n-1}h^{\frac{n}{n-1}}(s)ds=\frac{1}n h_0^{\frac{n}{n-1}},
$$
it follows that
\begin{equation}\label{p_over_r_limit}
\lim\limits_{r\to0^+}\frac{\varphi^\prime(r)}r=\lim\limits_{r\to0^+}C_n\left(\frac{1}{r^n}\int_0^r s^{n-1}h^{\frac{n}{n-1}}(s)ds\right)^{\frac{n-1}{n}}=\frac{h_0}{n-1}>0.
\end{equation}
In particular, $\varphi^\prime(r)>0$ for all sufficiently small $r>0$. Substituting \eqref{p_over_r_limit} into \eqref{second_positive_part}, with $h^+=h$, gives
$$
\lim\limits_{r\to0^+}\left(\varphi^{\prime\prime}(r)+(n-2)\frac{\varphi^\prime(r)}r\right)=h_0.
$$
Hence
\begin{equation}\label{second_limit}
\lim\limits_{r\to0^+}\varphi^{\prime\prime}(r)=h_0-(n-2)\frac{h_0}{n-1}=\frac{h_0}{n-1}.
\end{equation}
Since $\varphi^\prime(0)=0$, \eqref{p_over_r_limit} also gives $\varphi^{\prime\prime}(0)=\frac{h_0}{n-1}$. Together with \eqref{second_limit}, this proves (i)--(ii).

We next prove (iii). Since $f\equiv0$, $g(a)>0$, and $\varphi^\prime>0$ on $(0,R_0]$, there exists $R_1\in(0,R_0]$ such that
$$
h(r)=g(\varphi(r))(\varphi^\prime(r))^q>0,\quad r\in(0,R_1].
$$
Substituting this identity into \eqref{Volterra_differential}, multiplying by $(1-q)(r^{n-2}\varphi^\prime(r))^{-\frac{nq}{n-1}}$, and integrating from $0$ to $r$ give
\begin{equation}\label{degenerate_direct_identity}
\left(r^{n-2}\varphi^\prime(r)\right)^{\frac{n(1-q)}{n-1}}=n(1-q)(n-1)^{-\frac{n}{n-1}}\int_0^r s^{n-1+\frac{(2-n)nq}{n-1}}g(\varphi(s))^{\frac{n}{n-1}}ds.
\end{equation}
Since $n+\frac{(2-n)nq}{n-1}>0$ and $g(\varphi(r))\to g(a)$ as $r\to0^+$, \eqref{degenerate_direct_identity} yields
\begin{equation}\label{O_degenerate_1}
\varphi^\prime(r)=O\left(r^{\frac{1}{1-q}}\right),\quad\text{as }r\to0^+.
\end{equation}
Taking the logarithmic derivative of \eqref{degenerate_direct_identity} gives
$$
\frac{n(1-q)}{n-1}\left(n-2+\frac{r\varphi^{\prime\prime}(r)}{\varphi^\prime(r)}\right)=\frac{r^{n+\frac{(2-n)nq}{n-1}}g(\varphi(r))^{\frac{n}{n-1}}}{\displaystyle\int_0^r s^{n-1+\frac{(2-n)nq}{n-1}}g(\varphi(s))^{\frac{n}{n-1}}ds}\to n+\frac{(2-n)nq}{n-1},\quad\text{as }r\to0^+.
$$
Therefore,
\begin{equation}\label{O_degenerate_2}
\frac{r\varphi^{\prime\prime}(r)}{\varphi^\prime(r)}\to\frac{n-1}{n(1-q)}\left(n+\frac{(2-n)nq}{n-1}\right)-(n-2)=\frac{1}{1-q},\quad\text{as }r\to0^+.
\end{equation}
Since $\frac{1}{1-q}>1$, \eqref{O_degenerate_1} gives $\varphi^\prime(r)/r\to0$ as $r\to0^+$, and \eqref{O_degenerate_2} gives $\varphi^{\prime\prime}(r)\to0$ as $r\to0^+$. Hence $\varphi^{\prime\prime}(0)=0$ and $\varphi\in C^2[0,R]$.
This proves (iii).

We finally prove (iv).
Since $\varphi^\prime(r)\to0$ as $r\to0^+$ and $f$ and $g$ are continuous, there exist $R_*\in(0,R_0]$ and $M>0$ such that
$$
0<\varphi^\prime(r)\leq1,\quad |f(\varphi(r))|+|g(\varphi(r))|\leq M,\quad r\in(0,R_*].
$$
Since $q<0$, we have $(\varphi^\prime)^q\geq1$ on $(0,R_*]$, and hence
$$
0\leq h^+(r)\leq |f(\varphi(r))|+|g(\varphi(r))|(\varphi^\prime(r))^q\leq M(\varphi^\prime(r))^q,\quad r\in(0,R_*].
$$
Raising \eqref{modified_sODE} to the power $n$ and using $\varphi^{\prime\prime}\geq0$, we obtain
\begin{equation}\label{MA_second_estimate}
\left(\varphi^{\prime\prime}(r)\right)^{n-1}\leq\frac{r\,(h^+(r))^n}{(n-1)\varphi^\prime(r)}\leq\frac{M^n}{n-1}r(\varphi^\prime(r))^{nq-1},\quad r\in(0,R_*].
\end{equation}
Let $1<\delta<\frac{n(1-q)}{1-nq}$, $\theta:=\frac{\delta-1}{n-1}$, and $\sigma:=(nq-1)\theta$. Then $\sigma>-1$. For any $\varepsilon\in(0,R_*)$, \eqref{MA_second_estimate} gives
\begin{align*}
\int_\varepsilon^{R_*}|\varphi^{\prime\prime}(r)|^\delta dr&\leq\left(\frac{M^n}{n-1}\right)^\theta\int_\varepsilon^{R_*}r^\theta(\varphi^\prime(r))^\sigma\varphi^{\prime\prime}(r)dr\\
&\leq\frac{\left(\frac{M^n}{n-1}\right)^\theta R_*^\theta}{\sigma+1}\left[(\varphi^\prime(R_*))^{\sigma+1}-(\varphi^\prime(\varepsilon))^{\sigma+1}\right].
\end{align*}
Letting $\varepsilon\to0^+$ yields $\varphi^{\prime\prime}\in L^\delta(0,R_*)$. Moreover, since $\varphi^{\prime\prime}\geq0$ near the origin and $\varphi^\prime(0)=0$, we have $\int_0^{R_*}|\varphi^{\prime\prime}(r)|dr=\varphi^\prime(R_*)<+\infty$. Together with $\varphi\in C^2(0,R]$, this proves (iv).
\end{proof}

\begin{remark}\label{Rem2.2-iv}
For $q<0$, the one-dimensional Sobolev embedding applied to Lemma \ref{Lem2.2}~(iv) gives $\varphi\in C^{1,\alpha}[0,R]$ for any $0<\alpha<\frac{n-1}{n(1-q)}$. Under the hypotheses of Lemma \ref{Lem2.2}~(iv) and condition (A), the origin regularity improves in the following cases. If $g(a)=0$, then $f(a)>0$, \eqref{Volterra} gives $\varphi^\prime(r)/r\geq f(a)/(n-1)$, and the estimate \eqref{MA_second_estimate} yields
\begin{equation}\label{singular_regular_g_zero}
\varphi\in W_{\mathrm{loc}}^{2,\delta}[0,R],\quad 1\leq\delta<1-\frac{n-1}{nq}.
\end{equation}
If, in addition, $\ell:=\lim\limits_{r\to0^+}g(\varphi(r))(\varphi^\prime(r))^q\in\mathbb R$ and $f(a)+\ell>0$, then the proof of Lemma \ref{Lem2.2}~(i)--(ii) gives $\varphi\in C^2[0,R]$ and $\varphi^{\prime\prime}(0)=\frac{f(a)+\ell}{n-1}$. If $g(a)>0$, write $h(r)=(g(\varphi(r))+f(\varphi(r))(\varphi^\prime(r))^{-q})(\varphi^\prime(r))^q$. The coefficient in parentheses tends to $g(a)>0$ when $r\to0^+$. The integration and logarithmic differentiation used in the proof of Lemma \ref{Lem2.2}~(iii), with this coefficient in place of $g(\varphi)$, yield
\begin{equation}\label{O}
\varphi^\prime(r)=Dr^{\frac{1}{1-q}}(1+o(1)),\quad \varphi^{\prime\prime}(r)=\frac{D}{1-q}r^{\frac{q}{1-q}}(1+o(1)),\quad\text{as }r\to0^+,
\end{equation}where $D:=\left[g(a)\left(\frac{1}{1-q}+n-2\right)^{-\frac{n-1}{n}}(n-1)^{-\frac{1}n}\right]^{\frac{1}{1-q}}$,
and therefore
\begin{equation}\label{singular_regular_g_positive}
\varphi\in W_{\mathrm{loc}}^{2,\delta}[0,R],\quad 1\leq\delta<1-\frac{1}q.
\end{equation}
\end{remark}
\begin{remark}
The derivative conditions in Lemma \ref{Lem2.2}~(iii)--(iv) are stated for solutions of \eqref{mVolterra}. Theorems \ref{Th3.1} and \ref{Th3.2} construct local solutions satisfying $\varphi^\prime(r)>0$ and $\varphi^{\prime\prime}(r)\geq0$ for $r\in(0,R]$. Thus no separate monotonicity or convexity hypothesis is required in the main existence theorems.
\end{remark}

We next establish monotonicity and convexity for solutions of the auxiliary equation \eqref{mVolterra}. The positivity conclusion in part (iii) will justify returning to the original equation when $g$ changes sign.

\begin{lemma}\label{Lem2.3}
Assume that $f,g\in C(\mathbb R)$ and $q\in\mathbb R$.
Let $\varphi\in C^1[0,+\infty)\cap C^2(0,+\infty)$ satisfy \eqref{mVolterra}, $\varphi(0)=a>0$, and $\varphi^\prime(0)=0$.
If $q>0$, assume that either $f(a)>0$ or $\varphi^\prime(r)>0$ on $(0,r_0]$ for some $r_0>0$. If $q=0$, assume that $f(a)+g(a)>0$.
Then the following assertions hold.
\begin{enumerate}[label=(\roman*)]
\item $\varphi^\prime(r)>0$ and $\varphi^{\prime\prime}(r)+(n-2)\frac{\varphi^\prime(r)}r\geq0$ for all $r>0$.

\item If $f$ and $g$ are non-decreasing, then $\varphi^{\prime\prime}(r)\geq0$ for all $r>0$.

\item Assume that $q>0$, $f$ and $g$ are non-decreasing, and $f(a)>0$. Then $h(r)>0$ for $r>0$ and
\begin{equation}\label{iii}
\varphi^\prime(r)<\left(\frac{f(\varphi(r))}{g^-(\varphi(r))}\right)^{\frac{1}q},\quad r>0,
\end{equation}
where $g^-(t):=\max\{-g(t),0\}$ and the right-hand side is understood as $+\infty$ when $g^-(\varphi(r))=0$.
\end{enumerate}
\end{lemma}

\begin{proof}
We first prove (i). Since the integrand in \eqref{mVolterra} is non-negative, we have $\varphi^\prime(r)\geq0$ for $r>0$.
If $q>0$ and $f(a)>0$, then $\lim\limits_{r\to0^+}h(r)=f(a)>0$.
If $q=0$, then $\lim\limits_{r\to0^+}h(r)= f(a)+g(a)>0$. In either case, $h^+(r)>0$ for all sufficiently small $r>0$. Hence, for any $r>0$,
$$
\int_0^r s^{n-1}h^+(s)^{\frac{n}{n-1}}ds>0,
$$
and \eqref{mVolterra} gives $\varphi^\prime(r)>0$ for all $r>0$. If $q>0$ and $\varphi^\prime(r)>0$ on $(0,r_0]$, then \eqref{mVolterra} at $r=r_0$ gives
$$
\int_0^{r_0}s^{n-1}h^+(s)^{\frac{n}{n-1}}ds>0,
$$
which yields $\varphi^\prime(r)>0$ for all $r\geq r_0$, and therefore for all $r>0$. If $q<0$, the fact that \eqref{mVolterra} is well-defined implies $\varphi^\prime(r)\neq0$ for any $r>0$. Combining this with $\varphi^\prime(r)\geq0$, we again obtain $\varphi^\prime(r)>0$ for all $r>0$.
Consequently, \eqref{second_positive_part} is valid for any $r>0$. This proves (i).

We prove (ii). If $n=2$, the conclusion follows directly from \eqref{second_positive_part}. Let $n\geq3$ and define
$$
\mathcal F(y,z):=\bigl(y+(n-2)z\bigr)^{\frac{n-1}{n}}\bigl((n-1)z\bigr)^{\frac{1}n}
$$
for $z>0$ and $y+(n-2)z\geq0$. The function $\mathcal F$ is increasing in both variables. Suppose that $\varphi^{\prime\prime}(\tau)<0$ for some $\tau>0$. Since $\varphi^\prime(0)=0<\varphi^\prime(\tau)$ and $\varphi^{\prime\prime}(\tau)<0$, $\sigma:=\min\{r\in(0,\tau):\varphi^\prime(r)=\varphi^\prime(\tau)\}$ is well-defined. Then $0<\sigma<\tau$, $\varphi^{\prime\prime}(\sigma)\geq0$, $\varphi^\prime(\sigma)=\varphi^\prime(\tau)>0$, and $\frac{\varphi^\prime(\sigma)}{\sigma}>\frac{\varphi^\prime(\tau)}{\tau}$.
Consequently,
\begin{equation}\label{F_comparison}
\mathcal F\left(\varphi^{\prime\prime}(\tau),\frac{\varphi^\prime(\tau)}{\tau}\right)<\mathcal F\left(\varphi^{\prime\prime}(\sigma),\frac{\varphi^\prime(\sigma)}{\sigma}\right).
\end{equation}
On the other hand, (i) gives $\varphi(\tau)>\varphi(\sigma)$. Since $f$ and $g$ are non-decreasing and $\varphi^\prime(\tau)=\varphi^\prime(\sigma)$, we have $h(\tau)\geq h(\sigma)$ and hence $h^+(\tau)\geq h^+(\sigma)$. This contradicts \eqref{F_comparison} in view of \eqref{modified_sODE}. Therefore $\varphi^{\prime\prime}(r)\geq0$ for $r>0$.

We finally prove (iii). If $n\geq3$, (i) and (ii) give
$$
\frac{\varphi^\prime(r)}r>0,\quad \varphi^{\prime\prime}(r)+(n-2)\frac{\varphi^\prime(r)}r>0,\quad r>0.
$$
Thus \eqref{modified_sODE} implies $h^+(r)>0$, and hence $h(r)>0$ for all $r>0$.

It remains to consider $n=2$. Since $\lim\limits_{r\to0^+}h(r)=f(a)>0$, any connected component $(r_1,r_2)$ of $\{r>0:h(r)<0\}$ has $r_1>0$ and $h(r_1)=0$.
On this interval, it follows from \eqref{second_positive_part} that $\varphi^{\prime\prime}=0$. Hence $\varphi^\prime$ is constant on $(r_1,r_2)$, and the monotonicity of $f$ and $g$ implies that $h=f(\varphi)+g(\varphi)(\varphi^\prime)^q$ is non-decreasing. This contradicts $h(r_1)=0$ and $h<0$ on $(r_1,r_2)$. Therefore $h(r)\geq0$ for $r>0$.
Suppose that $h$ vanishes, and let $r_0>0$ be its first zero. Then $h>0$ on $(0,r_0)$. By \eqref{second_positive_part} with $n=2$, we have
$$
\varphi^{\prime\prime}(r)=\frac{rh^2(r)}{\varphi^\prime(r)},\quad 0<r<r_0.
$$
Since $f(\varphi(r_0))\geq f(a)>0$, the identity $h(r_0)=0$ gives $g(\varphi(r_0))<0$. For $r<r_0$, the monotonicity of $f$ and $g$ yields
$$
f(\varphi(r))+g(\varphi(r))(\varphi^\prime(r_0))^q\leq f(\varphi(r_0))+g(\varphi(r_0))(\varphi^\prime(r_0))^q=0.
$$
Since $\varphi^\prime$ is non-decreasing, it follows that, for $r<r_0$ sufficiently close to $r_0$,
$$
0<h(r)\leq-g(\varphi(r))\bigl((\varphi^\prime(r_0))^q-(\varphi^\prime(r))^q\bigr)\leq C\bigl(\varphi^\prime(r_0)-\varphi^\prime(r)\bigr)\leq C\int_r^{r_0}h^2(s)ds.
$$
Taking the maximum over $[r_1,r_0]$, with $r_1<r_0$ sufficiently close to $r_0$, gives
$$
0<\max_{[r_1,r_0]}h\leq C(r_0-r_1)\left(\max_{[r_1,r_0]}h\right)^2,
$$
This is impossible as $r_1\to r_0^-$. Hence $h(r)>0$ for $r>0$. If $g(\varphi(r))<0$, then $g^-(\varphi(r))(\varphi^\prime(r))^q<f(\varphi(r))$, which gives \eqref{iii}. Otherwise, the stated convention applies. This proves (iii).
\end{proof}

\begin{remark}\label{Rem2.2}
Let $R>0$ and let $\varphi\in C^1[0,R]\cap C^2(0,R]$ satisfy \eqref{mVolterra}, $\varphi(0)=a>0$, and $\varphi^\prime(0)=0$. Assume condition (A), with no additional hypothesis if $q\leq0$; if $q>0$, assume either $f(a)>0$ or $\varphi^\prime>0$ on $(0,r_0]$ for some $r_0\in(0,R]$. Alternatively, assume condition (B) and $q>0$. Then the proof of Lemma \ref{Lem2.3}~(i)--(ii), applied on $(0,R]$, gives
$$
\varphi^\prime(r)>0,\quad \varphi^{\prime\prime}(r)\geq0,\quad r\in(0,R].
$$
Under condition (A), $f,g\geq0$, $f(t)+g(t)>0$ for $t>0$, and $\varphi^\prime>0$ give $h>0$. Under condition (B), Lemma \ref{Lem2.3}~(iii) gives $h>0$. Consequently, $h^+=h$, and \eqref{mVolterra} and \eqref{modified_sODE} reduce to \eqref{Volterra} and \eqref{sODE}, respectively. By \eqref{2.1}, the associated radial function $v(x):=\varphi(|x|)$ is strictly $(n-1)$-convex in $B_R\setminus\{0\}$.

Suppose that $q<0$ and that the hypotheses of Lemma \ref{Lem2.2}~(iv) hold. Since $\varphi^\prime$ is absolutely continuous and $\varphi^\prime(0)=0$, Jensen's inequality and integration in polar coordinates give
$$
\int_0^R\left|\frac{\varphi^\prime(r)}r\right|^\delta r^{n-1}dr\leq\frac{R^{n-1}}{n-1}\int_0^R|\varphi^{\prime\prime}(r)|^\delta dr.
$$
The radial Hessian formula therefore yields
$$
v\in C^1(B_R)\cap C^2(B_R\setminus\{0\})\cap W_{\mathrm{loc}}^{2,\delta}(B_R),\quad 1\leq\delta<\frac{n(1-q)}{1-nq}.
$$
If $g(a)=0$ and the finite-limit assumptions in Remark \ref{Rem2.2-iv} hold, then $v\in C^2(B_R)$.
If $g(a)>0$, \eqref{O} yields $|D^2v(x)|\leq C|x|^{\frac{q}{1-q}}$ for all sufficiently small $|x|>0$, which gives the improved local Sobolev regularity
$v\in W_{\mathrm{loc}}^{2,\delta}(B_R)$ for $\delta\in[1,n(1-\frac{1}q))$.
\end{remark}

We conclude this section with a parameter comparison for exponentially weighted Keller--Osserman integrals. We state it in a general form that also applies to the companion $k$-Hessian paper.

\begin{lemma}\label{Lem2.4}
Let $F\in C([0,+\infty),[0,+\infty))$ be non-decreasing, with $F(t)>0$ for $t>0$.
\begin{enumerate}[label=(\roman*)]
\item Let $P\in C((0,+\infty),[0,+\infty))$ be non-increasing. For $\lambda>0$, $b\geq0$, and $t>b$, set
\begin{equation}\label{weighted_J_definition}
J_{\lambda,b}^-(t):=\int_b^t F(s)e^{-\lambda\int_s^tP(\tau)d\tau}ds.
\end{equation}
When $b=0$, this denotes $\lim\limits_{b\to0^+}J_{\lambda,b}^-(t)$. The function $J_{\lambda,b}^-$ is positive and non-decreasing on $(b,+\infty)$, and, for $0<\mu<\lambda$,
\begin{equation}\label{weighted_parameter_comparison}
J_{\lambda,b}^-(t)\leq J_{\mu,b}^-(t)\leq\frac{\lambda}{\mu} J_{\lambda,b}^-(t),\quad t>b.
\end{equation}
Consequently, for any $\theta>0$ and $A>b$, the convergence or divergence of
$$
\int_A^{+\infty} \frac{dt}{\left(J_{\lambda,b}^-(t)\right)^\theta}
$$
is independent of $\lambda>0$ and of the finite lower limits $b$ and $A$.

\item Let $G\in C([0,+\infty),[0,+\infty))$ be non-decreasing. Fix $a>0$ and $0<\theta<1$ and, for $\lambda>0$, consider
\begin{equation}\label{positive_weighted_J_definition}
\int_a^{+\infty} \frac{dt}{\left(\int_a^t F(s)e^{\lambda\int_s^tG(\tau)d\tau}ds\right)^\theta}.
\end{equation}
If $G\equiv0$ on $[a,+\infty)$, the integral is independent of $\lambda$; otherwise, it is finite for every $\lambda>0$. In particular, its convergence or divergence is independent of $\lambda>0$.
\end{enumerate}
\end{lemma}

\begin{proof}
We first prove (i) for $b>0$. Fix $0<\mu<\lambda$. For $r\in[b,t]$, set $s=t-\frac{\mu}{\lambda}(t-r)$. Since $P$ is non-increasing and $s\geq r$,
$$
\int_s^tP(\tau)d\tau\leq\frac{\mu}{\lambda}\int_r^tP(\tau)d\tau.
$$
Changing variables and using the monotonicity of $F$, we obtain
\begin{align*}
J_{\lambda,b}^-(t)&\geq\frac{\mu}{\lambda}\int_b^t F\left(t-\frac{\mu}{\lambda}(t-r)\right)e^{-\lambda\int_{t-\frac{\mu}{\lambda}(t-r)}^tP(\tau)d\tau}dr\\
&\geq\frac{\mu}{\lambda}\int_b^t F(r)e^{-\mu\int_r^tP(\tau)d\tau}dr=\frac{\mu}{\lambda} J_{\mu,b}^-(t).
\end{align*}
The inequality $J_{\lambda,b}^-(t)\leq J_{\mu,b}^-(t)$ follows from $P\geq0$. Letting $b\to0^+$ proves the same comparison for $b=0$; the limit is finite since the exponential factor is at most $1$. This proves independence of $\lambda$ for every $\theta>0$.

For $t>b$, changing variables gives
$$
J_{\lambda,b}^-(t)=\int_0^{t-b}F(t-s)e^{-\lambda\int_0^sP(t-\tau)d\tau}ds.
$$
The integrand is non-decreasing in $t$, since $F$ is non-decreasing and $P$ is non-increasing. Thus $J_{\lambda,b}^-$ is non-decreasing. For $t>2b$, translating the integral over $[t/2,t]$ by $t/2$ also gives
$$
J_{\lambda,0}^-(t/2)\leq J_{\lambda,b}^-(t)\leq J_{\lambda,0}^-(t).
$$
These inequalities, followed by the substitution $t=2s$ in the outer integral, prove independence of $b$. Independence of $A>b$ follows from positivity and continuity on $(b,+\infty)$.

We next prove (ii). Since $F(a)>0$ and $\theta<1$, the integral \eqref{positive_weighted_J_definition} is finite near $t=a$. If $G\not\equiv0$ on $[a,+\infty)$, choose $T>a$ such that $G(T)>0$. For $t\geq T$ and $\lambda>0$, the monotonicity of $G$ gives
$$
\int_a^t F(s)e ^{\lambda\int_s^tG(\tau)d\tau}ds\geq e ^{\lambda G(T)(t-T)}\int_a^T F(s)ds.
$$
Thus the outer integrand is bounded by a constant multiple of $e^{-\theta\lambda G(T)(t-T)}$, and the integral is finite at infinity. If $G\equiv0$ on $[a,+\infty)$, the weight equals one.
\end{proof}

\begin{remark}\label{Rem2.4}
	In this paper, we use $F=f^{\frac{n}{n-1}}$ and $\theta=\frac{n-1}{2n-1}$. Under condition (B), $P=P_n$ defined in \eqref{Pn} is nonnegative and non-increasing for $0<q\leq2-\frac{1}{n}$, since $g^-$ is non-increasing and $f$ is positive and non-decreasing on $(0,+\infty)$. 
	For the critical case with $f(t)>0$ for $t>0$, the positive-weight formula \eqref{positive_weighted_J_definition} uses $G=g^{\frac{n}{n-1}}$ under condition (A), or $G=(g^+)^{\frac{n}{n-1}}$ under condition (B). The companion $k$-Hessian paper uses the same lemma with $F=f$, $\theta=\frac{1}{2}$, and its corresponding weight coefficients.
\end{remark}

\section{Local Existence for the Radial Cauchy Problem}\label{Section3}
This section establishes local solvability under conditions (A) and (B). For each $a>0$, we seek a solution on $[0,R]$ for some $R>0$ of the Cauchy problem
\begin{equation}\label{CP}
\left\{
\begin{aligned}
&\varphi^\prime(r)=C_nr^{2-n}\left(\int_0^r s^{n-1}h^{\frac{n}{n-1}}(s)ds\right)^{\frac{n-1}{n}},\quad r\in(0,R],\\
&\varphi(0)=a,\quad \varphi^\prime(0)=0,
\end{aligned}
\right.
\end{equation}
where $C_n=\frac{n^{\frac{n-1}{n}}}{n-1}$ and $h(r)=f(\varphi(r))+g(\varphi(r))|\varphi^\prime(r)|^q$. We adapt the Euler method to an integral functional that depends on both $\varphi$ and $\varphi^\prime$. The regular case carries out the compactness argument; the singular case requires a power transformation and positive lower slope bounds. In the gradient-degenerate case, the divergence structure reduces \eqref{CP} to the scalar Cauchy problem treated in \cite[Theorem 1.1]{JTX2026}.

We now establish the local existence of \eqref{CP} near $r=0$ under condition (A).

\begin{theorem}[Local Existence under Condition (A)]\label{Th3.1}
Assume that $f$ and $g$ satisfy condition (A), and let $q\in\mathbb R$. If $q>0$, assume additionally that either $f(t)>0$ for all $t>0$, or that $0<q<1$ and $f\equiv0$. Then, for any initial value $a>0$, there exists $R>0$ such that \eqref{CP} admits a positive solution $\varphi\in C^1[0,R]\cap C^2(0,R]$ satisfying
$\varphi^\prime(r)>0$ and $\varphi^{\prime\prime}(r)\geq0$ for all $r\in(0,R]$. Moreover,
\begin{itemize}
\item[(i)] If $q\geq0$, then $\varphi\in C^2[0,R]$.
\item[(ii)] If $q<0$, then $\varphi\in C^2(0,R]\cap W^{2,\delta}[0,R]$ for $\delta\in[1,\frac{n(1-q)}{1-nq})$.
\end{itemize}
\end{theorem}
\begin{proof}
Let constants $l,\eta>0$ be determined later. Define
$$
\mathcal R:=\left\{(r,v)\in[0,l]\times PC^1[0,l]:v(0)=a,\ a\leq v(s)\leq a+\eta,\ 0\leq v^\prime(s)\leq1,\ s\in(0,l)\right\},
$$
where $PC^1[0,l]$ denotes the space of continuous and piecewise continuously differentiable functions on $[0,l]$.
We treat the three cases separately.

\textbf{Case I (Regular case): $q=0$, or $q>0$ and $f>0$.}

\textit{Step 1: We define the Euler break line $\psi$.}
Fix $\eta>0$.
Choose $l>0$ such that
\begin{equation}\label{Euler_choice_I}
\frac{f(a+\eta)+g(a+\eta)}{n-1}l\leq1,\quad \frac{f(a+\eta)+g(a+\eta)}{n-1}l^2\leq\eta.
\end{equation}
For $(r,v)\in\mathcal R$, define
\begin{equation}\label{Euler_functional_I}
\mathcal G_1(r,v):=C_nr^{2-n}\left(\int_0^r s^{n-1}\left[f(v(s))+g(v(s))|v^\prime(s)|^q\right]^\frac n{n-1} ds\right)^{\frac{n-1}n},
\end{equation}where $C_n=\frac{n^{\frac{n-1}{n}}}{n-1}$, and set $\mathcal G_1(0,v):=0$.
If $q=0$, condition (A) gives $f(t)+g(t)>0$ for $t>0$. If $q>0$, then $f(t)>0$ for $t>0$. Hence
$$
\mathcal G_1(r,v)>0,\quad r>0,\quad (r,v)\in\mathcal R.
$$
Let $M:=f(a+\eta)+g(a+\eta)$.
Since $f$ and $g$ are nonnegative and non-decreasing, $0\leq v^\prime\leq1$, and $a\leq v\leq a+\eta$, we have
$$
0\leq f(v(s))+g(v(s))|v^\prime(s)|^q\leq M,\quad s\in(0,r).
$$
It follows that
\begin{equation}\label{Euler_bound_I}
0\leq\mathcal G_1(r,v)\leq \frac{M}{n-1}r,\quad (r,v)\in\mathcal R.
\end{equation}
In particular, \eqref{Euler_choice_I} gives
$$
0\leq\mathcal G_1(r,v)\leq1,\quad (r,v)\in\mathcal R.
$$
We define the Euler break line $\psi(\cdot)$ on $[0,l]$ by
\begin{equation}\label{Euler_break_line_I}
\begin{cases}
\psi(r)=a,&0\leq r\leq r_1,\\
\psi(r)=\psi(r_{i-1})+\mathcal G_1(r_{i-1},\psi)(r-r_{i-1}),&r_{i-1}<r\leq r_i,\quad i=2,\ldots,m,
\end{cases}
\end{equation}where $0=r_0<r_1<\cdots<r_m=l$ and $m\in \mathbb{N}$.
The construction is recursive, since $\mathcal G_1(r_{i-1},\psi)$ depends only on $\psi$ on $[0,r_{i-1}]$. Induction using \eqref{Euler_bound_I} and \eqref{Euler_choice_I} gives
$$
0\leq\psi^\prime(r)\leq\frac{Ml}{n-1}\leq1,\quad r\in(0,l),
$$
and
$$
0\leq\psi(r)-a=\int_0^r\psi^\prime(s)ds\leq\frac{Ml^2}{n-1}\leq\eta,\quad r\in[0,l].
$$
Thus $(r,\psi)\in\mathcal R$.

\textit{Step 2: We prove that $\psi$ is an $\varepsilon$-approximate solution.} To do this, we need to prove that for any small $\varepsilon>0$, we can choose appropriate points $\left\{r_i\right\}_{i=1, \cdots, m}$, such that
\begin{equation}\label{Euler_approximate}
|\psi^\prime(r)-\mathcal G_1(r,\psi)|<\varepsilon,\quad r\in(0,l).
\end{equation}
By \eqref{Euler_bound_I}, we can choose $\bar r\in(0,l)$ such that $\frac{M\bar r}{n-1}<\frac{\varepsilon}{2}$.
Set $r_1=\bar{r}$. Then
$$
|\psi^\prime(r)-\mathcal G_1(r,\psi)|
=\mathcal G_1(r, \psi)<\frac{\varepsilon}{2}, \quad 0 \leq r<\bar{r}. 
$$
For $\bar{r} \leq r \leq l$, without loss of generality, we assume $r_{i-1} \leq r \leq r_i$ for some $i=2, \cdots, m$.
Then
$$
\begin{aligned}
& |\psi^\prime(r)-\mathcal G_1(r,\psi)|\\
=& \left|\mathcal G_1\left(r_{i-1}, \psi\right)-\mathcal G_1(r, \psi)\right| \\
\leq& C_n\left|r_{i-1}^{-\frac{n(n-2)}{n-1}}\int_0^{r_{i-1}}s^{n-1}[f(\psi)+g(\psi)|\psi^\prime|^q]^{\frac{n}{n-1}}ds-r^{-\frac{n(n-2)}{n-1}}\int_0^rs^{n-1}[f(\psi)+g(\psi)|\psi^\prime|^q]^{\frac{n}{n-1}}ds\right|^{\frac{n-1}n}.
\end{aligned}
$$
Splitting the second integral at $r_{i-1}$ and using $0\leq f(\psi)+g(\psi)|\psi^\prime|^q\leq M$, we find
$$
\begin{aligned}
|\psi^\prime(r)-\mathcal G_1(r,\psi)|
\leq&\frac{M}{n-1}\left(r_{i-1}^n\left|r_{i-1}^{-\frac{n(n-2)}{n-1}}-r^{-\frac{n(n-2)}{n-1}}\right|+r^{-\frac{n(n-2)}{n-1}}(r^n-r_{i-1}^n)\right)^{\frac{n-1}n}\\
\leq&\frac{M}{n-1}\left(l^n\left|r_{i-1}^{-\frac{n(n-2)}{n-1}}-r^{-\frac{n(n-2)}{n-1}}\right|+\bar r^{-\frac{n(n-2)}{n-1}}(r^n-r_{i-1}^n)\right)^{\frac{n-1}n}.
\end{aligned}
$$
Since $\bar r\leq r_{i-1}<r\leq l$, the mean value theorem yields
\begin{align*}
\left|r_{i-1}^{-\frac{n(n-2)}{n-1}}-r^{-\frac{n(n-2)}{n-1}}\right|\leq\frac{n(n-2)}{n-1}\bar r^{-\frac{n(n-2)}{n-1}-1}(r-r_{i-1}),\quad
r^n-r_{i-1}^n\leq nl^{n-1}(r-r_{i-1}).
\end{align*}
For the above $\varepsilon>0$, we take
$$
0<\delta_1<\frac{\bar r^{\frac{n(n-2)}{n-1}+1}}{\frac{n(n-2)}{n-1}l^n+nl^{n-1}\bar r}\left(\frac{(n-1)\varepsilon}{2M}\right)^{\frac{n}{n-1}},
$$
Then for $\max\limits_{2\leq i\leq m}(r_i-r_{i-1})<\delta_1$, we have
$$
|\psi^\prime(r)-\mathcal G_1(r,\psi)|<\frac{\varepsilon}{2},\quad r\in(\bar r,l)
$$
Together with the estimate on $(0,\bar r)$, this proves \eqref{Euler_approximate}.

\textit{Step 3: We obtain a solution as a limit of Euler break lines.} Let $\varepsilon_j\to0^+$, and let $\psi_j$ be an $\varepsilon_j$-approximate solution constructed in \eqref{Euler_break_line_I}. By Step 2, we have
\begin{equation}\label{varepsilon_j}
|\psi_j^\prime(r)-\mathcal G_1(r,\psi_j)|<\varepsilon_j,\quad r\in(0,l).
\end{equation}
Since $0\leq\psi_j^\prime\leq1$ and $\psi_j(0)=a$, we have
$$
|\psi_j(r^{\prime\prime})-\psi_j(r^\prime)|\leq|r^{\prime\prime}-r^\prime|,\quad r^\prime,r^{\prime\prime}\in[0,l].
$$
Thus $\{\psi_j\}$ is uniformly bounded and uniformly Lipschitz continuous on $[0,l]$.
By the Arzelà--Ascoli theorem, there exists a uniformly convergent subsequence, which we still denote by $\psi_j$.
Assume that
$$
\lim\limits_{j \rightarrow+\infty} \psi_j=\varphi\quad\text{uniformly on }[0,l].
$$ In particular, $\varphi\in C[0,l]$, $\varphi(0)=a$, and $a\leq\varphi\leq a+\eta$.

Set
$$
\Delta_j(r):=\psi_j^\prime(r)-\mathcal G_1(r,\psi_j),\quad r\in(0,l).
$$
By \eqref{varepsilon_j}, $\|\Delta_j\|_{L^\infty(0,l)}\leq\varepsilon_j$.
Moreover, \eqref{Euler_approximate} holds for arbitrary $r_{i-1},r\in[\bar r,l]$, independently of the partition and of $j$. It gives a common modulus of continuity on every $[\bar r,l]$. Together with \eqref{Euler_bound_I} near $0$, it shows that $\{\mathcal G_1(\cdot,\psi_j)\}$ is uniformly bounded and equicontinuous on $[0,l]$.
Passing to a further subsequence, we obtain
$$
\lim\limits_{j \rightarrow+\infty} \mathcal G_1(\cdot,\psi_j)= p\quad\text{uniformly on }[0,l]
$$
for some $p\in C[0,l]$. Consequently,
$$
\|\psi_j^\prime-p\|_{L^\infty(0,l)}\leq\|\Delta_j\|_{L^\infty(0,l)}+\|\mathcal G_1(\cdot,\psi_j)-p\|_{C[0,l]}\to0.
$$
Integrating the definition of $\Delta_j$ and letting $j\to+\infty$, we obtain
$$
\varphi(r)=\lim_{j\to+\infty}\psi_j(r)=a+\lim_{j\to+\infty}\int_0^r\left[\mathcal G_1(s,\psi_j)+\Delta_j(s)\right]ds=a+\int_0^r p(s)ds,\quad r\in[0,l].
$$
Therefore $\varphi\in C^1[0,l]$ and $\varphi^\prime=p$. Moreover, $\varphi^\prime(0)=p(0)=0$.
The uniform convergence of $\psi_j$ and the $L^\infty$ convergence of $\psi_j^\prime$ imply
$$
f(\psi_j)+g(\psi_j)|\psi_j^\prime|^q\to f(\varphi)+g(\varphi)|\varphi^\prime|^q\quad\text{in }L^\infty(0,l).
$$
Passing to the limit in \eqref{Euler_functional_I} gives
$$
\varphi^\prime(r)=\mathcal G_1(r,\varphi),\quad r\in[0,l],\qquad \varphi(0)=a.
$$
In Case I, the integrand defining $\mathcal G_1(r,\varphi)$ is positive for $r>0$. Hence $\varphi^\prime(r)>0$ for $r\in(0,l]$. Differentiating the integral expression shows that $\varphi\in C^2(0,l]$ and satisfies \eqref{CP}. Lemma \ref{Lem2.2}~(i)--(ii) gives $\varphi\in C^2[0,l]$, and Lemma \ref{Lem2.3}~(ii) gives $\varphi^{\prime\prime}(r)\geq0$ for $r\in(0,l]$. This completes Case I.

\textbf{Case II (Singular case): $q<0$.}

\textit{Step 1: We define the Euler break line $\psi$.} Fix $\eta>0$, set $\alpha:=\frac{1}{1-q}\in(0,1)$, and use $M=f(a+\eta)+g(a+\eta)$ as in Case I. Choose $l\in(0,1]$ such that
\begin{equation}\label{Euler_choice_II}
\left[M(\alpha+n-2)^{-\frac{n-1}{n}}(n-1)^{-\frac{1}n}\right]^\alpha l^\alpha\leq\min\{1,\eta\}.
\end{equation}
For $(r,v)\in\mathcal R$, define
\begin{equation}\label{Euler_functional_II}
\mathcal G_2(r,v):=\widetilde C_nr^{2-n}\left(\int_0^r s^{n-1+\frac{(2-n)nq}{n-1}}\left[f(v(s))|v^\prime(s)|^{-q}+g(v(s))\right]^{\frac{n}{n-1}}ds\right)^{\frac{n-1}{n(1-q)}},
\end{equation}
where $\widetilde C_n:=\left(n(1-q)(n-1)^{-\frac{n}{n-1}}\right)^{\frac{n-1}{n(1-q)}}$, and set $\mathcal G_2(0,v):=0$. Since $-q>0$, the transformed integrand is continuous even where $v^\prime=0$. For $0\leq v^\prime\leq1$, we have $0\leq f(v)|v^\prime|^{-q}+g(v)\leq M$. Moreover,
$$
\int_0^r s^{n-1+\frac{(2-n)nq}{n-1}}ds=\frac{(n-1)\alpha}{n(\alpha+n-2)}r^{\frac{n(\alpha+n-2)}{(n-1)\alpha}}.
$$
Substitution into \eqref{Euler_functional_II} gives
\begin{equation}\label{Euler_upper_II}
0\leq\mathcal G_2(r,v)\leq\left[M(\alpha+n-2)^{-\frac{n-1}{n}}(n-1)^{-\frac{1}n}\right]^\alpha r^\alpha,\quad (r,v)\in\mathcal R.
\end{equation}

Suppose first that $g(a)>0$. Since $g(v(s))\geq g(a)$, the same calculation gives
\begin{equation}\label{Euler_lower_II_positive}
\mathcal G_2(r,v)\geq\left[g(a)(\alpha+n-2)^{-\frac{n-1}{n}}(n-1)^{-\frac{1}n}\right]^\alpha r^\alpha.
\end{equation}
Define
\begin{equation}\label{Euler_break_line_II_positive}
\begin{cases}
\psi(r)=a,&0\leq r\leq r_1,\\
\psi(r)=\psi(r_{i-1})+\mathcal G_2(r_{i-1},\psi)(r-r_{i-1}),&r_{i-1}<r\leq r_i,\quad i=2,\ldots,m,
\end{cases}
\end{equation}
where $0=r_0<r_1<\cdots<r_m=l$. By \eqref{Euler_choice_II} and \eqref{Euler_upper_II}, $0\leq\psi^\prime\leq1$ and $0\leq\psi-a\leq l\eta\leq\eta$. Thus $(r,\psi)\in\mathcal R$.

Suppose next that $g(a)=0$. Condition (A) gives $f(a)>0$.
A constant initial segment would keep the transformed Euler approximation constant, so we impose a linear lower slope bound.
Choose $l>0$ so small that, in addition to \eqref{Euler_choice_II}, $\frac{f(a)}{2(n-1)}l\leq1$ and $l\leq\eta$. If $v^\prime(s)\geq\frac{f(a)}{2(n-1)}s$, then monotonicity of $f$ gives
$$
f(v(s))|v^\prime(s)|^{-q}+g(v(s))\geq f(a)\left(\frac{f(a)}{2(n-1)}\right)^{-q}s^{-q}.
$$
Since
$$
\int_0^r s^{n-1+\frac{(2-n)nq}{n-1}-\frac{nq}{n-1}}ds=\frac{r^{n(1-q)}}{n(1-q)},
$$
equation \eqref{Euler_functional_II} yields
\begin{equation}\label{Euler_lower_II_zero}
\mathcal G_2(r,v)\geq\left(\frac{f(a)}{n-1}\right)^\alpha\left(\frac{f(a)}{2(n-1)}\right)^{-q\alpha}r=2^\alpha\frac{f(a)}{2(n-1)}r\geq\frac{f(a)}{2(n-1)}r.
\end{equation}
Define
\begin{equation}\label{Euler_break_line_II_zero}
\begin{cases}
\psi(r)=a+\frac{f(a)}{2(n-1)}r_1r,&0\leq r\leq r_1,\\
\psi(r)=\psi(r_{i-1})+\max\left\{\mathcal G_2(r_{i-1},\psi),\frac{f(a)}{2(n-1)}r_i\right\}(r-r_{i-1}),&r_{i-1}<r\leq r_i,\quad i=2,\ldots,m,
\end{cases}
\end{equation}
It follows from \eqref{Euler_break_line_II_zero} and \eqref{Euler_lower_II_zero} that $\psi^\prime(r)\geq\frac{f(a)}{2(n-1)}r$. Moreover, \eqref{Euler_upper_II} and the choice of $l$ give $\psi^\prime\leq1$ and $a\leq\psi(r)\leq a+l\leq a+\eta$. Thus $(r,\psi)\in\mathcal R$.

\textit{Step 2: We prove that $\psi$ is an $\varepsilon$-approximate solution.} Fix $\varepsilon>0$. If $g(a)>0$, choose $\bar r\in(0,l)$ such that
$$
\left[M(\alpha+n-2)^{-\frac{n-1}{n}}(n-1)^{-\frac{1}n}\right]^\alpha\bar r^\alpha<\frac{\varepsilon}{2}
$$
and set $r_1=\bar r$. Since $\psi^\prime=0$ on $(0,\bar r)$, \eqref{Euler_upper_II} gives
$$
|\psi^\prime(r)-\mathcal G_2(r,\psi)|=\mathcal G_2(r,\psi)<\frac{\varepsilon}{2},\quad 0<r<\bar r.
$$
Let $\bar r\leq r_{i-1}<r\leq r_i\leq l$ for some $i=2,\ldots,m$. As in Case I, we split the integral at $r_{i-1}$ and use $|x^\theta-y^\theta|\leq|x-y|^\theta$ with $\theta=\frac{(n-1)\alpha}{n}\in(0,1)$. The bound $0\leq f(\psi)|\psi^\prime|^{-q}+g(\psi)\leq M$ then gives
$$
\begin{aligned}
&|\mathcal G_2(r,\psi)-\mathcal G_2(r_{i-1},\psi)|\\
&\leq\widetilde C_nM^\alpha\Bigg[\left(r_{i-1}^{-\frac{n(n-2)}{(n-1)\alpha}}-r^{-\frac{n(n-2)}{(n-1)\alpha}}\right)\int_0^{r_{i-1}}s^{n-1+\frac{(2-n)nq}{n-1}}ds+r^{-\frac{n(n-2)}{(n-1)\alpha}}\int_{r_{i-1}}^r s^{n-1+\frac{(2-n)nq}{n-1}}ds\Bigg]^\frac{(n-1)\alpha}{n}\\
&\leq \widetilde C_nM^\alpha\left(1+\frac{n-2}{\alpha+n-2}\right)^\frac{(n-1)\alpha}{n} l^{\frac{\alpha}{n}}(r-r_{i-1})^\frac{(n-1)\alpha}{n}.
\end{aligned}
$$
Choose
$$
0<\delta_2<\frac{1}{1+\frac{n-2}{\alpha+n-2}}\left(\frac{\varepsilon}{2\widetilde C_nM^\alpha l^{\frac{\alpha}{n}}}\right)^{\frac{n}{(n-1)\alpha}}.
$$
If $\max\limits_{2\leq i\leq m}(r_i-r_{i-1})<\delta_2$, the last estimate gives $|\psi^\prime(r)-\mathcal G_2(r,\psi)|<\varepsilon/2$ on the remaining intervals when $g(a)>0$.

If $g(a)=0$, choose $\bar r\in(0,l)$ such that
$$
\frac{f(a)}{2(n-1)}\bar r+\left[M(\alpha+n-2)^{-\frac{n-1}{n}}(n-1)^{-\frac{1}n}\right]^\alpha\bar r^\alpha<\frac{\varepsilon}{2}
$$
and set $r_1=\bar r$. On $(r_{i-1},r_i)$, \eqref{Euler_lower_II_zero} gives
$$
0\leq\psi^\prime(r)-\mathcal G_2(r_{i-1},\psi)\leq\frac{f(a)}{2(n-1)}(r_i-r_{i-1}).
$$
Choose
$$
0<\delta_3<\min\left\{\frac{1}{1+\frac{n-2}{\alpha+n-2}}\left(\frac{\varepsilon}{4\widetilde C_nM^\alpha l^{\frac{\alpha}{n}}}\right)^{\frac{n}{(n-1)\alpha}},\frac{\varepsilon(n-1)}{2f(a)}\right\}.
$$
If $\max\limits_{2\leq i\leq m}(r_i-r_{i-1})<\delta_3$, the same difference estimate yields
$$
|\psi^\prime(r)-\mathcal G_2(r,\psi)|\leq\frac{f(a)}{2(n-1)}(r_i-r_{i-1})+\widetilde C_nM^\alpha\left(1+\frac{n-2}{\alpha+n-2}\right)^{\frac{(n-1)\alpha}{n}}l^{\frac{\alpha}{n}}(r-r_{i-1})^{\frac{(n-1)\alpha}{n}}<\frac{\varepsilon}{2}.
$$
Therefore, in both subcases, the partition can be chosen so that
\begin{equation}\label{Euler_error_II}
|\psi^\prime(r)-\mathcal G_2(r,\psi)|<\varepsilon,\quad r\in(0,l).
\end{equation}

\textit{Step 3: We obtain a solution as a limit of Euler break lines.}
The difference estimate for $\mathcal G_2$ in Step 2 holds for arbitrary $0<r_{i-1}<r\leq l$, independently of the partition. Together with \eqref{Euler_upper_II}, it gives a common modulus of continuity on $[0,l]$, so the compactness argument of Case I applies.
Along a subsequence, $\psi_j\to\varphi$ and $\mathcal G_2(\cdot,\psi_j)\to p$ uniformly, while \eqref{Euler_error_II} gives $\psi_j^\prime\to p$ in $L^\infty(0,l)$. Hence $\varphi\in C^1[0,l]$, $\varphi^\prime=p$, $\varphi(0)=a$, and $\varphi^\prime(0)=0$. Since $-q>0$, passage to the limit in \eqref{Euler_functional_II} yields
\begin{equation}\label{Euler_limit_II}
\varphi^\prime(r)=\widetilde C_nr^{2-n}\left(\int_0^r s^{n-1+\frac{(2-n)nq}{n-1}}\left[f(\varphi(s))|\varphi^\prime(s)|^{-q}+g(\varphi(s))\right]^{\frac{n}{n-1}}ds\right)^{\frac{n-1}{n(1-q)}},\quad r\in(0,l].
\end{equation}
If $g(a)>0$, \eqref{Euler_lower_II_positive} gives $\varphi^\prime(r)\geq Dr^\alpha>0$, where $D$ is as in Remark \ref{Rem2.2-iv}. If $g(a)=0$, the lower slope bound gives $\varphi^\prime(r)\geq\frac{f(a)}{2(n-1)}r>0$.
Multiplying \eqref{Euler_limit_II} by $r^{n-2}$ and raising both sides to the power $\frac{n(1-q)}{n-1}$, we obtain
$$
\left(r^{n-2}\varphi^\prime(r)\right)^{\frac{n(1-q)}{n-1}}=n(1-q)(n-1)^{-\frac{n}{n-1}}\int_0^r s^{n-1+\frac{(2-n)nq}{n-1}}\left[f(\varphi(s))(\varphi^\prime(s))^{-q}+g(\varphi(s))\right]^{\frac{n}{n-1}}ds.
$$
Since $\varphi^\prime>0$, the integrand is continuous on $(0,l]$, and the identity implies that $\varphi^\prime\in C^1(0,l]$. Differentiating this identity and using $\varphi^\prime(r)>0$, we find
$$
\begin{aligned}
\frac{d}{dr}\left[\left(r^{n-2}\varphi^\prime(r)\right)^{\frac{n}{n-1}}\right]
&=n(n-1)^{-\frac{n}{n-1}}r^{n-1+\frac{(2-n)nq}{n-1}}\left(r^{n-2}\varphi^\prime(r)\right)^{\frac{nq}{n-1}}\left[f(\varphi(r))(\varphi^\prime(r))^{-q}+g(\varphi(r))\right]^{\frac{n}{n-1}}\\
&=n(n-1)^{-\frac{n}{n-1}}r^{n-1}\left[f(\varphi(r))+g(\varphi(r))(\varphi^\prime(r))^q\right]^{\frac{n}{n-1}}.
\end{aligned}
$$This recovers \eqref{Volterra_differential}. Integration from $0$ to $r$ then gives \eqref{CP}.
Thus $\varphi\in C^2(0,l]$, Lemma \ref{Lem2.3}~(ii) gives $\varphi^{\prime\prime}\geq0$, and Lemma \ref{Lem2.2}~(iv) gives
$\varphi\in W^{2,\delta}[0,l]$ for $1\leq\delta<\frac{n(1-q)}{1-nq}$.
This completes Case II.

\textbf{Case III (Degenerate case): $0<q<1$ and $f\equiv0$.} Under condition (A), we have $g(t)>0$ for all $t>0$. We construct an immediately increasing branch. On any such branch, $h(r)=g(\varphi(r))(\varphi^\prime(r))^q$, and multiplying \eqref{Volterra_differential} by $(1-q)(r^{n-2}\varphi^\prime(r))^{-\frac{nq}{n-1}}$ gives
$$
\frac{d}{dr}\left[\left(r^{n-2}\varphi^\prime(r)\right)^{\frac{n(1-q)}{n-1}}\right]=n(1-q)(n-1)^{-\frac{n}{n-1}}r^{n-1+\frac{(2-n)nq}{n-1}}g^{\frac{n}{n-1}}(\varphi(r)).
$$
The initial condition $\varphi^\prime(0)=0$ therefore suggests the transformed Cauchy problem
\begin{equation}\label{degenerate_CP}
\left\{
\begin{aligned}
&\varphi^\prime(r)=\widetilde C_nr^{2-n}\left(\int_0^r s^{n-1+\frac{(2-n)nq}{n-1}}g^{\frac{n}{n-1}}(\varphi(s))ds\right)^{\frac{n-1}{n(1-q)}},\quad r>0,\\
&\varphi(0)=a,\quad \varphi^\prime(0)=0,
\end{aligned}
\right.
\end{equation}where $\widetilde C_n:=\left[n(1-q)(n-1)^{-\frac{n}{n-1}}\right]^{\frac{n-1}{n(1-q)}}$.
To apply \cite[Theorem 1.1]{JTX2026}, denote the parameter called $q$ in that theorem by $\rho$, and set
$$
\rho=n-2,\quad \theta=\frac{n(1-q)}{n-1},\quad \tau=n+\frac{(2-n)nq}{n-1},\quad C=\widetilde C_n.
$$
Then
$$
\tau-\rho\theta=\frac{n}{n-1}>0,\quad \tau-\theta(\rho+1)=\frac{nq}{n-1}>0.
$$
Thus \cite[Theorem 1.1(i)]{JTX2026} gives $R_*>0$ and a solution $\varphi\in C^2[0,R_*)$ with $\varphi(r)\geq a$, $\varphi^\prime(r)>0$, $\varphi^{\prime\prime}(0)=0$, and $\varphi^{\prime\prime}(r)>0$ for $r\in(0,R_*)$. Since $\varphi^\prime>0$, multiplying \eqref{degenerate_CP} by $r^{n-2}$, taking the power $\frac{n(1-q)}{n-1}$, and differentiating recover \eqref{Volterra_differential}.
Integration from $0$ to $r$ gives \eqref{Volterra}, so this immediately increasing branch solves \eqref{CP}.
This completes Case III.

Taking $R=l$ in Cases I and II and choosing any $R\in(0,R_*)$ in Case III completes the proof.
\end{proof}

We next treat condition (B), where $g$ may change sign. Since $f(a)>0$, the negative part of $g$ can be controlled by restricting the slopes of the Euler break lines. This allows us to use the same construction as above.

\begin{theorem}[Local Existence under Condition (B)]\label{Th3.2}
Assume that $f$ and $g$ satisfy condition (B), and let $q>0$. Then, for any initial value $a>0$, there exists $R>0$ such that \eqref{CP} admits a positive solution $\varphi\in C^2[0,R]$ satisfying
$\varphi^\prime(r)>0$ and $\varphi^{\prime\prime}(r)\geq0$ for all $r\in(0,R]$.
\end{theorem}

\begin{proof}
We use the improved Euler method and divide the proof into three steps.

\textit{Step 1: We define the Euler break line $\psi$.} Since $q>0$ and $f(a)>0$, choose $P\in(0,1]$ such that
$$
g^-(a)P^q\leq\frac{f(a)}{2},
$$
where $g^-(a):=\max\{-g(a),0\}$. Fix $\eta>0$, set $M:=f(a+\eta)+\max\{g(a+\eta),0\}P^q$,
and choose $l>0$ such that
\begin{equation}\label{Euler_choice_Pl}
\frac{M}{n-1}l\leq P,\quad \frac{M}{n-1}l^2\leq\eta.
\end{equation}
Define
$$
\mathcal R:=\left\{(r,v)\in[0,l]\times PC^1[0,l]:v(0)=a,\ a\leq v(s)\leq a+\eta,\ 0\leq v^\prime(s)\leq P,\ s\in(0,l)\right\}.
$$
For $(r,v)\in\mathcal R$ and $r>0$, let
\begin{equation}\label{Euler_functional_B}
\mathcal G_B(r,v):=C_nr^{2-n}\left(\int_0^r s^{n-1}\left[f(v(s))+g(v(s))|v^\prime(s)|^q\right]^{\frac{n}{n-1}}ds\right)^{\frac{n-1}n},
\end{equation}
where $C_n=\frac{n^{\frac{n-1}{n}}}{n-1}$, and set $\mathcal G_B(0,v):=0$. Since $f$ and $g$ are non-decreasing, the restrictions defining $\mathcal R$ give
$$
\frac{f(a)}{2}\leq f(v(s))+g(v(s))|v^\prime(s)|^q\leq M,\quad s\in(0,l).
$$
It follows that
\begin{equation}\label{Euler_functional_B_bound}
\frac{f(a)}{2(n-1)}r\leq\mathcal G_B(r,v)\leq\frac{M}{n-1}r,\quad (r,v)\in\mathcal R.
\end{equation}
For a partition $0=r_0<r_1<\cdots<r_m=l$, define
\begin{equation}\label{Euler_break_line_B}
\begin{cases}
\psi(r)=a,&0\leq r\leq r_1,\\
\psi(r)=\psi(r_{i-1})+\mathcal G_B(r_{i-1},\psi)(r-r_{i-1}),&r_{i-1}<r\leq r_i,\quad i=2,\ldots,m.
\end{cases}
\end{equation}
By \eqref{Euler_functional_B_bound} and \eqref{Euler_choice_Pl},
$$
0\leq\psi^\prime(r)\leq\frac{Ml}{n-1}\leq P,\quad r\in(0,l),
$$
and
$$
0\leq\psi(r)-a=\int_0^r\psi^\prime(s)ds\leq\frac{Ml^2}{n-1}\leq\eta,\quad r\in[0,l].
$$
Thus $(r,\psi)\in\mathcal R$.

\textit{Step 2: We prove that $\psi$ is an $\varepsilon$-approximate solution.} Fix $\varepsilon>0$. By \eqref{Euler_functional_B_bound}, choose $\bar r\in(0,l)$ such that $\frac{M\bar r}{n-1}<\frac{\varepsilon}{2}$, and take $r_1=\bar r$. On $(0,\bar r)$, we have $\psi^\prime=0$ and hence
$$
|\psi^\prime(r)-\mathcal G_B(r,\psi)|=\mathcal G_B(r,\psi)<\frac{\varepsilon}{2}.
$$
Let $\bar r\leq r_{i-1}<r<r_i\leq l$ for some $i=2,\ldots,m$. Splitting the integral at $r_{i-1}$ and applying the mean value theorem, we obtain
$$
\begin{aligned}
|\psi^\prime(r)-\mathcal G_B(r,\psi)|&=|\mathcal G_B(r_{i-1},\psi)-\mathcal G_B(r,\psi)|\\
&\leq\frac{M}{n-1}\left[r_{i-1}^n\left|r_{i-1}^{-\frac{n(n-2)}{n-1}}-r^{-\frac{n(n-2)}{n-1}}\right|+r^{-\frac{n(n-2)}{n-1}}(r^n-r_{i-1}^n)\right]^{\frac{n-1}n}\\
&\leq\frac{M}{n-1}\left[\bar r^{-\frac{n(n-2)}{n-1}-1}\left(\frac{n(n-2)}{n-1}l^n+nl^{n-1}\bar r\right)(r-r_{i-1})\right]^{\frac{n-1}n}.
\end{aligned}
$$
Choose $\delta_1>0$ such that
$$
\delta_1<\frac{\bar r^{\frac{n(n-2)}{n-1}+1}}{\frac{n(n-2)}{n-1}l^n+nl^{n-1}\bar r}\left(\frac{(n-1)\varepsilon}{2M}\right)^{\frac{n}{n-1}},
$$
and choose the remaining partition points so that $\max\limits_{2\leq i\leq m}(r_i-r_{i-1})<\delta_1$. The preceding estimate then gives $|\psi^\prime(r)-\mathcal G_B(r,\psi)|<\frac{\varepsilon}{2}$ on the remaining subintervals. Consequently,
\begin{equation}\label{Euler_approximation_B}
|\psi^\prime(r)-\mathcal G_B(r,\psi)|<\varepsilon,\quad r\in(0,l).
\end{equation}

\textit{Step 3: We obtain a solution as a limit of Euler break lines.} Let $\varepsilon_j\to0^+$ and let $\psi_j$ be the corresponding Euler break lines. By \eqref{Euler_functional_B_bound} and the estimates in Step 2, the compactness argument in Step 3 of the regular case under condition (A) applies. Thus, after passing to a subsequence, there exists $\varphi\in C^1[0,l]$ such that
$$
\|\psi_j-\varphi\|_{C[0,l]}+\|\psi_j^\prime-\varphi^\prime\|_{L^\infty(0,l)}\to0\quad\text{as }j\to+\infty,
$$
with $\varphi(0)=a$ and $\varphi^\prime(0)=0$. Since $q>0$ and $f,g$ are continuous, passing to the limit in \eqref{Euler_functional_B} and \eqref{Euler_approximation_B} gives $\varphi^\prime(r)=\mathcal G_B(r,\varphi)$ on $(0,l]$. Moreover, \eqref{Euler_functional_B_bound} yields
$$
\frac{f(a)}{2(n-1)}r\leq\varphi^\prime(r)\leq\frac{M}{n-1}r,\quad r\in(0,l].
$$
Differentiating the integral identity gives $\varphi\in C^2(0,l]$ and shows that $\varphi$ solves \eqref{CP}. Lemma \ref{Lem2.2}~(i) then yields $\varphi\in C^2[0,l]$ with $\varphi^{\prime\prime}(0)=\frac{f(a)}{n-1}>0$. Finally, Lemma \ref{Lem2.3}~(ii) gives $\varphi^{\prime\prime}\geq0$ on $(0,l]$. Taking $R=l$ completes the proof.
\end{proof}

\begin{remark}\label{Rem3.1}
Every local solution constructed above admits a maximal continuation on an interval $[0,R_\varphi)$, where $0<R_\varphi\leq+\infty$. The proof of Lemma \ref{Lem2.3}, applied on each finite subinterval, gives $\varphi^\prime>0$ and $\varphi^{\prime\prime}\geq0$ throughout this interval.

Suppose that $R_\varphi<+\infty$ and both $\varphi$ and $\varphi^\prime$ remain bounded as $r\to R_\varphi^-$. Their monotonicity gives finite limits $\lim\limits_{r\to R_\varphi^-}\varphi(r)\geq a$ and $\lim\limits_{r\to R_\varphi^-}\varphi^\prime(r)>0$. Moreover, $\lim\limits_{r\to R_\varphi^-}h(r)>0$ under condition \textup{(A)}, and the same conclusion under condition \textup{(B)} follows as in the proof of Lemma \ref{Lem2.3}~(iii). Thus \eqref{sODE} can be solved for $\varphi^{\prime\prime}$ with a continuous right-hand side near the limiting point, and Peano's theorem gives a continuation beyond $R_\varphi$, contradicting maximality. Consequently, if $R_\varphi<+\infty$, then
\begin{equation}\label{finite_radius_alternative}
\lim_{r\to R_\varphi^-}\varphi(r)=+\infty\quad\text{or}\quad\lim_{r\to R_\varphi^-}\varphi^\prime(r)=+\infty.
\end{equation}
\end{remark}

\section{Global Existence for the Radial Cauchy Problem}\label{Section4}

We derive two-sided derivative estimates for increasing convex radial solutions and use them to characterize global solvability of \eqref{CP}. Since these estimates depend on the sign of $g$, we treat conditions \textup{(A)} and \textup{(B)} separately. Throughout this section, we write $\theta_1:=\frac{n-1}{2n-1}$ and, for $q<2-\frac1n$, $\theta_2:=\frac{n-1}{2n-1-qn}$.

\subsection{Global Existence under Condition (A)}

We first work under condition (A). Raising \eqref{sODE} to the power $\frac{n}{n-1}$ leads naturally to estimates involving $f^{\frac{n}{n-1}}$ and $g^{\frac{n}{n-1}}$.

\begin{lemma}\label{Lemma4.2}
Assume that $f$ and $g$ satisfy condition (A) and that $q\leq2-\frac{1}n$. Let $R\in(0,+\infty]$ and suppose that $\varphi\in C^1[0,R)\cap C^2(0,R)$ is a solution of \eqref{sODE} satisfying $\varphi(0)=a>0$, $\varphi^\prime(0)=0$, $\varphi^\prime(r)>0$, and $\varphi^{\prime\prime}(r)\geq0$ for all $r\in(0,R)$.
Then, for all $r\in(0,R)$,
\begin{equation}\label{1prime_n_upper}
\varphi^{\prime}(r)\leq
\left\{
\begin{aligned}
&C_1\left[r^{\frac{1}{2n-1}}\left(\int_a^{\varphi(r)}f^{\frac{n}{n-1}}(t)dt\right)^{\theta_1}+r^{\frac{1}{2n-1-qn}}\left(\int_a^{\varphi(r)}g^{\frac{n}{n-1}}(t)dt\right)^{\theta_2}\right],&&q<2-\frac{1}n,\\ &\left(\Lambda_1r^{\frac{1}{n-1}}\int_a^{\varphi(r)}e^{\Lambda_1r^{\frac{1}{n-1}}\int_t^{\varphi(r)}g^{\frac{n}{n-1}}(s)ds}f^{\frac{n}{n-1}}(t)dt\right)^{\theta_1},&&q=2-\frac{1}n,
\end{aligned}
\right.
\end{equation}
and
\begin{equation}\label{1prime_n_lower}
\varphi^{\prime}(r)\geq
\left\{
\begin{aligned}
&C_2\left[r^{\frac{1}{2n-1}}\left(\int_a^{\varphi(r)}f^{\frac{n}{n-1}}(t)dt\right)^{\theta_1}+r^{\frac{1}{2n-1-qn}}\left(\int_a^{\varphi(r)}g^{\frac{n}{n-1}}(t)dt\right)^{\theta_2}\right],&&q<2-\frac{1}n,\\
&\left(\Gamma_1r^{\frac{1}{n-1}}\int_{\varphi(\frac{r}{2})}^{\varphi(r)}e^{\Gamma_1r^{\frac{1}{n-1}}\int_t^{\varphi(r)}g^{\frac{n}{n-1}}(s)ds}f^{\frac{n}{n-1}}(t)dt\right)^{\theta_1},&&q=2-\frac{1}n.
\end{aligned}
\right.
\end{equation}
Here $C_1,C_2>0$ are constants depending only on $n$ and $q$,  $\Lambda_1:=\frac{1}{\theta_1}\left(\frac{2}{n-1}\right)^{\frac{1}{n-1}}$, and  $\Gamma_1:=\frac{1}{\theta_1}(n-1)^{-\frac{n}{n-1}}2^{-\frac{1}{n-1}}$.

\end{lemma}

\begin{proof}
We first prove the upper bound \eqref{1prime_n_upper}. Since $\varphi^{\prime\prime}\geq0$, equation \eqref{sODE} gives
\begin{equation}\label{eq4.4}
\begin{aligned}
\varphi^{\prime\prime}(r)(\varphi^\prime(r))^{\frac{1}{n-1}}
&\leq\left(\frac{r}{n-1}\right)^{\frac{1}{n-1}}\left[f(\varphi(r))+g(\varphi(r))(\varphi^\prime(r))^q\right]^{\frac{n}{n-1}}\\
&\leq\left(\frac{2r}{n-1}\right)^{\frac{1}{n-1}}\left[f^{\frac{n}{n-1}}(\varphi(r))+g^{\frac{n}{n-1}}(\varphi(r))(\varphi^\prime(r))^{\frac{qn}{n-1}}\right].
\end{aligned}
\end{equation}

{\bfseries\boldmath Case 1: $q=0$.}
Multiplying \eqref{eq4.4} by $\varphi^\prime(r)$ and integrating from $0$ to $r$ yields
$$
(\varphi^\prime(r))^{\frac{2n-1}{n-1}}\leq\Lambda_1r^{\frac{1}{n-1}}\left[\int_a^{\varphi(r)}f^{\frac{n}{n-1}}(t)dt+\int_a^{\varphi(r)}g^{\frac{n}{n-1}}(t)dt\right].
$$
Taking the power $\theta_1$ proves \eqref{1prime_n_upper} for $q=0$.

{\bfseries\boldmath Case 2: $0<q<2-\frac{1}n$.}
Multiplying \eqref{eq4.4} by $\varphi^\prime(r)$ and integrating from $0$ to $r$, we obtain
$$
(\varphi^\prime(r))^{\frac{2n-1}{n-1}}\leq\Lambda_1r^{\frac{1}{n-1}}\int_a^{\varphi(r)}f^{\frac{n}{n-1}}(t)dt+\Lambda_1r^{\frac{1}{n-1}}(\varphi^\prime(r))^{\frac{qn}{n-1}}\int_a^{\varphi(r)}g^{\frac{n}{n-1}}(t)dt,
$$
where we used the monotonicity of $\varphi^\prime$, $\varphi(0)=a$, and $\varphi^\prime(0) = 0$.
Applying the $\varepsilon$-Young's inequality with $\varepsilon=\left(\frac{2n-1}{2qn}\right)^{\frac{qn}{2n-1}}$ to the second term on the right-hand side, we obtain
$$
\Lambda_1r^{\frac{1}{n-1}}(\varphi^\prime(r))^{\frac{qn}{n-1}}\int_a^{\varphi(r)}g^{\frac{n}{n-1}}(t)dt\leq\frac{1}{2}(\varphi^\prime(r))^{\frac{2n-1}{n-1}}+A_q\left(\Lambda_1r^{\frac{1}{n-1}}\int_a^{\varphi(r)}g^{\frac{n}{n-1}}(t)dt\right)^{\frac{2n-1}{2n-1-qn}},
$$where $A_q:=\frac{2n-1-qn}{2n-1}\left(\frac{2qn}{2n-1}\right)^{\frac{qn}{2n-1-qn}}$ is a constant depending only on $n$, $q$.
Consequently,
\begin{align*}
(\varphi^{\prime}(r))^{\frac{2n-1}{n-1}}
&\le 2\Lambda_1 r^{\frac{1}{n-1}} \int_a^{\varphi(r)} f^{\frac{n}{n-1}}(t) dt +2A_q \left(\Lambda_1 r^{\frac{1}{n-1}} \int_a^{\varphi(r)} g^{\frac{n}{n-1}}(t) dt \right)^{\frac{2n-1}{2n-1-qn}}.
\end{align*}
Raising both sides to the power $\theta_1\in(0,1)$ gives \eqref{1prime_n_upper} for $0<q<2-\frac{1}{n}$.

{\bfseries\boldmath Case 3: $q < 0$.} Set $\Lambda_2:=\frac{2n-1-qn}{n-1}\left(\frac{2}{n-1}\right)^{\frac{1}{n-1}}$.
Multiplying \eqref{eq4.4} by $(\varphi^\prime(r))^{1-\frac{qn}{n-1}}$, using the monotonicity of $(\varphi^\prime)^{-\frac{qn}{n-1}}$, and integrating from $0$ to $r$, we obtain
$$
(\varphi^\prime(r))^{\frac{2n-1-qn}{n-1}}\leq\Lambda_2r^{\frac{1}{n-1}}(\varphi^\prime(r))^{-\frac{qn}{n-1}}\int_a^{\varphi(r)}f^{\frac{n}{n-1}}(t)dt+\Lambda_2r^{\frac{1}{n-1}}\int_a^{\varphi(r)}g^{\frac{n}{n-1}}(t)dt.
$$
Applying the $\varepsilon$-Young's inequality with $\varepsilon=\left( \frac{2n-1-qn}{-2qn} \right)^{\frac{-qn}{2n-1-qn}}$ yields
$$
\Lambda_2r^{\frac{1}{n-1}}(\varphi^\prime(r))^{-\frac{qn}{n-1}}\int_a^{\varphi(r)}f^{\frac{n}{n-1}}(t)dt\leq\frac{1}{2}(\varphi^\prime(r))^{\frac{2n-1-qn}{n-1}}+B_q\left(\Lambda_2r^{\frac{1}{n-1}}\int_a^{\varphi(r)}f^{\frac{n}{n-1}}(t)dt\right)^{\frac{2n-1-qn}{2n-1}},
$$where $B_q:=\frac{2n-1}{2n-1-qn}\left(\frac{-2qn}{2n-1-qn}\right)^{-\frac{qn}{2n-1}}$ is a constant depending only on $n,q$.
Consequently,
\begin{align*}
(\varphi^{\prime}(r))^{\frac{2n-1-qn}{n-1}}
\le 2B_q \left( \Lambda_2 r^{\frac{1}{n-1}}\int_a^{\varphi(r)} f^{\frac{n}{n-1}}(t) dt \right)^{\frac{2n-1-qn}{2n-1}} + 2\Lambda_2 r^{\frac{1}{n-1}} \int_a^{\varphi(r)} g^{\frac{n}{n-1}}(t) dt.
\end{align*}
Raising both sides to the power of $\theta_2 \in (0,1)$ proves \eqref{1prime_n_upper} for $q < 0$.

{\bfseries\boldmath Case 4: $q = 2-\frac{1}{n}$.} Let $\psi(s):=(\varphi^\prime(s))^{\frac{2n-1}{n-1}}$. For $s\in(0,r]$, inequality \eqref{eq4.4} and $s^{\frac{1}{n-1}}\leq r^{\frac{1}{n-1}}$ imply
$$
\psi^\prime(s)-\Lambda_1r^{\frac{1}{n-1}}g^{\frac{n}{n-1}}(\varphi(s))\varphi^\prime(s)\psi(s)\leq\Lambda_1r^{\frac{1}{n-1}}f^{\frac{n}{n-1}}(\varphi(s))\varphi^\prime(s).
$$
Multiplying both sides by $e^{-\Lambda_1 r^{\frac{1}{n-1}} \int_a^{\varphi(s)} g^{\frac{n}{n-1}}(\tau) d \tau}$, we get
\begin{equation*}
\left( \psi(s) e^{-\Lambda_1 r^{\frac{1}{n-1}} \int_a^{\varphi(s)} g^{\frac{n}{n-1}}(\tau) d \tau} \right)^{\prime} \le \Lambda_1 r^{\frac{1}{n-1}} e^{-\Lambda_1 r^{\frac{1}{n-1}} \int_a^{\varphi(s)} g^{\frac{n}{n-1}}(\tau) d \tau} f^{\frac{n}{n-1}}(\varphi(s)) \varphi^{\prime}(s).
\end{equation*}
Integrating from $0$ to $r$, using $\psi(0) = (\varphi^{\prime}(0))^{\frac{2n-1}{n-1}}=0$, and multiplying by $e^{\Lambda_1 r^{\frac{1}{n-1}} \int_a^{\varphi(r)} g^{\frac{n}{n-1}}(\tau) d \tau}$, we obtain
\begin{equation*}
\psi(r) \le \Lambda_1 r^{\frac{1}{n-1}} \int_a^{\varphi(r)} e^{\Lambda_1 r^{\frac{1}{n-1}} \int_t^{\varphi(r)} g^{\frac{n}{n-1}}(\tau) d \tau} f^{\frac{n}{n-1}}(t) dt.
\end{equation*}
Substituting back $\psi(r) = (\varphi^{\prime}(r))^{\frac{2n-1}{n-1}}$ and raising both sides to the power of $\theta_1$ gives \eqref{1prime_n_upper} for $q = 2-\frac{1}{n}$.

We next prove \eqref{1prime_n_lower}. Raising \eqref{sODE} to the power $\frac{n}{n-1}$ gives
\begin{equation}\label{eq4.exp}
\varphi^{\prime\prime}(r)(\varphi^\prime(r))^{\frac{1}{n-1}}+\frac{n-2}{r}(\varphi^\prime(r))^{\frac{n}{n-1}}=\left(\frac{r}{n-1}\right)^{\frac{1}{n-1}}\left[f(\varphi(r))+g(\varphi(r))(\varphi^\prime(r))^q\right]^{\frac{n}{n-1}},\quad r>0.
\end{equation}We split the proof into two cases depending on $q$.

{\bfseries\boldmath Case 1: $q<2-\frac{1}n$.}
Multiplying \eqref{eq4.exp} by $\varphi^\prime(r)$ yields
\begin{equation}\label{eq4.linear}
\left((\varphi^\prime)^{\frac{2n-1}{n-1}}\right)^\prime+\frac{\alpha}{r}(\varphi^\prime)^{\frac{2n-1}{n-1}}\geq\widetilde C_1r^{\frac{1}{n-1}}\left[f^{\frac{n}{n-1}}(\varphi)\varphi^\prime+g^{\frac{n}{n-1}}(\varphi)(\varphi^\prime)^{\frac{qn}{n-1}+1}\right],
\end{equation}
where $\alpha:=\frac{(n-2)(2n-1)}{n-1}\geq0$ and $\widetilde C_1:=\frac{2n-1}{n-1}\left(\frac{1}{n-1}\right)^{\frac{1}{n-1}}$.
Since both terms on the right-hand side are non-negative, \eqref{eq4.linear} implies
\begin{align}
\left( (\varphi^{\prime}(r))^{\frac{2n-1}{n-1}} \right)^{\prime} + \frac{\alpha}{r}(\varphi^{\prime}(r))^{\frac{2n-1}{n-1}} &\ge\widetilde C_1 r^{ \frac{1}{n-1}} f^{\frac{n}{n-1}}(\varphi(r)) \varphi^\prime(r),\quad r>0,\label{eq4.8}\\
\left( (\varphi^{\prime}(r))^{\frac{2n-1}{n-1}} \right)^{\prime} + \frac{\alpha}{r}(\varphi^{\prime}(r))^{\frac{2n-1}{n-1}} &\ge \widetilde C_1 r^{ \frac{1}{n-1}} g^{\frac{n}{n-1}}(\varphi(r)) (\varphi^{\prime}(r))^{\frac{qn}{n-1} + 1},\quad r>0.\label{4.8}
\end{align}
Multiplying \eqref{eq4.8} by $r^\alpha$ and integrating over $[\frac{r}{2},r]$, we obtain
\begin{align*}
r^\alpha (\varphi^{\prime}(r))^{\frac{2n-1}{n-1}}
\ge\widetilde C_1 \int_{\frac{r}{2}}^r s^{2n-3} f^{\frac{n}{n-1}}(\varphi(s)) \varphi^\prime(s) ds \ge \widetilde C_1 \left(\frac{r}{2}\right)^{2n-3} \int_{\varphi(\frac{r}{2})}^{\varphi(r)} f^{\frac{n}{n-1}}(t) dt,\quad r>0.
\end{align*}
The convexity of $\varphi$ gives $\varphi(\frac{r}{2})-a\leq\varphi(r)-\varphi(\frac{r}{2})$. Together with the monotonicity of $f$, this implies
$$
\int_a^{\varphi(r)}f^{\frac{n}{n-1}}(t)dt\leq2\int_{\varphi(\frac{r}{2})}^{\varphi(r)}f^{\frac{n}{n-1}}(t)dt.
$$
Therefore,
\begin{equation}\label{lower_f_2}
\begin{split}
\varphi^{\prime}(r)
&\ge \left(\widetilde C_1 \left(\frac{1}{2}\right)^{2n-3}\right)^{\theta_1} r^{\frac{1}{2n-1}} \left( \int_{\varphi(\frac{r}{2})}^{\varphi(r)} f^{\frac{n}{n-1}}(t) dt \right)^{\theta_1}\\
&\ge \Gamma_2^{\theta_1} r^{\frac{1}{2n-1}} \left( \int_a^{\varphi(r)}f^{\frac{n}{n-1}}(t)dt \right)^{\theta_1},
\end{split}
\end{equation}where $\Gamma_2:=\widetilde C_1 2^{-(2n-2)}$.
Multiplying \eqref{4.8} by $\frac{2n-1-qn}{2n-1}(\varphi^{\prime}(r))^{-\frac{qn}{n-1}}$ gives
\begin{equation*}
\left((\varphi^{\prime}(r))^{\frac{2n-1-qn}{n-1}}\right)^\prime + \frac{\beta}{r} (\varphi^{\prime}(r))^{\frac{2n-1-qn}{n-1}} \ge \widetilde C_2 r^{\frac{1}{n-1}} g^{\frac{n}{n-1}}(\varphi(r)) \varphi^\prime(r),
\end{equation*}
where $\beta=\frac{(n-2)(2n-1-qn)}{n-1}\geq0$ and $\widetilde C_2=\frac{2n-1-qn}{n-1}\left(\frac{1}{n-1}\right)^{\frac{1}{n-1}}$.
Multiplying by $r^\beta$ and integrating over $[\frac{r}{2},r]$, we obtain
\begin{align*}
r^\beta (\varphi^{\prime}(r))^{\frac{2n-1-qn}{n-1}} \ge \widetilde C_2 \int_{\frac{r}{2}}^r s^{\beta + \frac{1}{n-1}} g^{\frac{n}{n-1}}(\varphi(s)) \varphi^\prime(s) ds \ge \widetilde C_2 \left(\frac{r}{2}\right)^{\beta + \frac{1}{n-1}} \int_{\varphi(\frac{r}{2})}^{\varphi(r)} g^{\frac{n}{n-1}}(t) dt,\quad r>0.
\end{align*}
The same convexity argument, together with the monotonicity of $g$, gives
\begin{equation}\label{lower_g_2}
\begin{split}
\varphi^{\prime}(r)
&\ge \left(\widetilde C_2 \left(\frac{1}{2}\right)^{\beta + \frac{1}{n-1}} \right)^{\theta_2} r^{\frac{1}{2n-1-qn}} \left( \int_{\varphi(\frac{r}{2})}^{\varphi(r)} g^{\frac{n}{n-1}}(t) dt \right)^{\theta_2}\\
&\ge \Gamma_3^{\theta_2} r^{\frac{1}{2n-1-qn}} \left( \int_a^{\varphi(r)}g^{\frac{n}{n-1}}(t)dt\right)^{\theta_2},
\end{split}
\end{equation}
where $\Gamma_3:=\widetilde C_2 2^{-\beta-\frac{n}{n-1}}$. Combining \eqref{lower_f_2} and \eqref{lower_g_2} proves \eqref{1prime_n_lower} for $q<2-\frac{1}n$ with
$$
C_2:=\frac{1}{2}\min\{\Gamma_2^{\theta_1},\Gamma_3^{\theta_2}\}>0.
$$

{\bfseries\boldmath Case 2: $q = 2-\frac{1}{n}$.}
Equation \eqref{sODE} is equivalent to
\begin{equation}\label{eq4.div}
\left[\left(r^{n-2}\varphi^\prime(r)\right)^{\frac{n}{n-1}}\right]^\prime = \frac{n}{(n-1)^{\frac{n}{n-1}}} r^{n-1} \left[ f(\varphi(r))+g(\varphi(r))(\varphi^\prime(r))^q \right]^{\frac{n}{n-1}},\quad r>0.
\end{equation}
Since $\varphi$, $\varphi^\prime$, $f$, and $g$ are non-decreasing, integrating \eqref{eq4.div} from $0$ to $r$ gives
\begin{align*}
\left(r^{n-2}\varphi^\prime(r)\right)^{\frac{n}{n-1}}
&= \frac{n}{(n-1)^{\frac{n}{n-1}}} \int_0^r s^{n-1} \left[ f(\varphi(s))+g(\varphi(s))(\varphi^\prime(s))^q \right]^{\frac{n}{n-1}} ds \\
&\le \frac{r^n}{(n-1)^{\frac{n}{n-1}}} \left[ f(\varphi(r))+g(\varphi(r))(\varphi^\prime(r))^q \right]^{\frac{n}{n-1}}.
\end{align*}
Hence
\begin{equation}\label{upper_grad}
(\varphi^\prime(r))^{\frac{n}{n-1}} \le \left(\frac{r}{n-1} \right)^{\frac{n}{n-1}}\left[ f(\varphi(r))+g(\varphi(r))(\varphi^\prime(r))^q \right]^{\frac{n}{n-1}}.
\end{equation}
Substituting \eqref{upper_grad} into \eqref{eq4.exp}, we obtain
\begin{align*}
\varphi^{\prime\prime}(r)(\varphi^\prime(r))^{\frac{1}{n-1}}
&\ge \left(\left(\frac{r}{n-1}\right)^{\frac{1}{n-1}} - \frac{n-2}{r} \left(\frac{r}{n-1} \right)^{\frac{n}{n-1}}\right) \left[ f(\varphi(r))+g(\varphi(r))(\varphi^\prime(r))^q \right]^{\frac{n}{n-1}}\\
&=\left(\frac{1}{n-1}\right)^\frac{n}{n-1} r^{\frac{1}{n-1}} \left[ f(\varphi(r))+g(\varphi(r))(\varphi^\prime(r))^q \right]^{\frac{n}{n-1}}\\
&\geq \left(\frac{1}{n-1}\right)^\frac{n}{n-1} r^{\frac{1}{n-1}} \left[ f^{\frac{n}{n-1}}(\varphi(r)) + g^{\frac{n}{n-1}}(\varphi(r)) (\varphi^{\prime}(r))^{\frac{qn}{n-1}} \right].
\end{align*}
Multiplying by $\varphi^\prime(r)$ gives
\begin{equation}\label{eq4.5}
\left( (\varphi^{\prime}(r))^{\frac{2n-1}{n-1}} \right)^{\prime}
\ge \widetilde C_3r^{\frac{1}{n-1}} \left[ f^{\frac{n}{n-1}}(\varphi(r)) \varphi^{\prime}(r) + g^{\frac{n}{n-1}}(\varphi(r)) (\varphi^{\prime}(r))^{\frac{qn}{n-1} + 1} \right],\quad r>0,
\end{equation}
where
$
\widetilde C_3=
\frac{2n-1}{n-1}
\left(\frac{1}{n-1}\right)^{\frac{n}{n-1}}.
$
Let $\psi(s) = (\varphi^{\prime}(s))^{\frac{2n-1}{n-1}}$. For $s\in[\frac{r}{2},r]$, we have
$s^{\frac{1}{n-1}}\geq\left(\frac{r}{2}\right)^{\frac{1}{n-1}}$.
Therefore,
\begin{equation*}
\psi^{\prime}(s) -\widetilde C_3 \left(\frac{r}{2}\right)^{\frac{1}{n-1}} g^{\frac{n}{n-1}}(\varphi(s)) \varphi^{\prime}(s) \psi(s) \ge\widetilde C_3 \left(\frac{r}{2}\right)^{\frac{1}{n-1}} f^{\frac{n}{n-1}}(\varphi(s)) \varphi^{\prime}(s).
\end{equation*}
Multiplying by $e^{-\widetilde C_3(\frac{r}{2})^{\frac{1}{n-1}} \int_a^{\varphi(s)} g^{\frac{n}{n-1}}(\tau) d\tau}$, we deduce
\begin{equation*}
\left( \psi(s) e^{-\widetilde C_3(\frac{r}{2})^{\frac{1}{n-1}} \int_a^{\varphi(s)} g^{\frac{n}{n-1}}(\tau) d\tau} \right)^{\prime} \ge\widetilde C_3 \left(\frac{r}{2}\right)^{\frac{1}{n-1}} e^{-\widetilde C_3(\frac{r}{2})^{\frac{1}{n-1}} \int_a^{\varphi(s)} g^{\frac{n}{n-1}}(\tau) d\tau} f^{\frac{n}{n-1}}(\varphi(s)) \varphi^{\prime}(s).
\end{equation*}
Integrating over $[\frac{r}{2},r]$ gives
\begin{align*}
\psi(r) e^{-\widetilde C_3(\frac{r}{2})^{\frac{1}{n-1}} \int_a^{\varphi(r)} g^{\frac{n}{n-1}}(\tau) d\tau} \ge\widetilde C_3 \left(\frac{r}{2}\right)^{\frac{1}{n-1}} \int_{\frac{r}{2}}^r e^{-\widetilde C_3(\frac{r}{2})^{\frac{1}{n-1}} \int_a^{\varphi(s)} g^{\frac{n}{n-1}}(\tau) d\tau} f^{\frac{n}{n-1}}(\varphi(s)) \varphi^{\prime}(s) ds.
\end{align*}
Multiplying both sides by $e^{\widetilde C_3(\frac{r}{2})^{\frac{1}{n-1}} \int_a^{\varphi(r)} g^{\frac{n}{n-1}}(\tau) d\tau}$, we obtain
\begin{align*}
\psi(r)
&\ge\widetilde C_3 \left(\frac{r}{2}\right)^{\frac{1}{n-1}} \int_{\varphi(\frac{r}{2})}^{\varphi(r)} e^{\widetilde C_3(\frac{r}{2})^{\frac{1}{n-1}} \int_t^{\varphi(r)} g^{\frac{n}{n-1}}(\tau) d\tau} f^{\frac{n}{n-1}}(t) dt\\
&=\Gamma_1r^{\frac{1}{n-1}}\int_{\varphi(\frac{r}{2})}^{\varphi(r)}e^{\Gamma_1r^{\frac{1}{n-1}}\int_t^{\varphi(r)}g^{\frac{n}{n-1}}(\tau)d\tau}f^{\frac{n}{n-1}}(t)dt.
\end{align*}
Substituting back $\psi(r) = (\varphi^{\prime}(r))^{\frac{2n-1}{n-1}}$ and taking the power $\theta_1$ proves \eqref{1prime_n_lower} for $q = 2-\frac{1}{n}$ and completes the proof.
\end{proof}

\begin{theorem}[Global Existence under Condition (A)]\label{Th4.2}
Assume that $f$ and $g$ satisfy condition (A), and let $q<2-\frac{1}n$. If $q>0$, assume additionally that either $f(t)>0$ for all $t>0$, or that $0<q<1$ and $f\equiv0$. For any initial value $a>0$, the Cauchy problem \eqref{CP} admits a positive global solution $\varphi\in C^1[0,+\infty)\cap C^2(0,+\infty)$ satisfying $\varphi^\prime(r)>0$ and $\varphi^{\prime\prime}(r)\geq0$ for all $r>0$ if and only if condition \eqref{KO_n} holds. Furthermore,
\begin{itemize}
\item[(i)] if $0\leq q<2-\frac{1}n$, then $\varphi\in C^2[0,+\infty)$;
\item[(ii)] if $q<0$, then $\varphi\in W_{\mathrm{loc}}^{2,\delta}[0,+\infty)$ for $1\leq\delta<\frac{n(1-q)}{1-nq}$.
\end{itemize}
\end{theorem}

\begin{proof}
\underline{\textit{Sufficiency.}} Assume that \eqref{KO_n} holds. Suppose, by contradiction, that for some initial value $a>0$, \eqref{CP} has no global solution with the stated properties. By Theorem \ref{Th3.1}, \eqref{CP} has a positive local solution with $\varphi^\prime>0$ and $\varphi^{\prime\prime}\geq0$. Let $[0,R)$ be a maximal continuation of this solution, as in Remark \ref{Rem3.1}. For every $\bar R<R$, the solution belongs to $C^1[0,\bar R]\cap C^2(0,\bar R]$ and satisfies \eqref{CP}. The integral in \eqref{CP} is positive near the origin by the local construction and remains positive as $r$ increases.
Hence $\varphi^\prime(r)>0$ and $\varphi(r)\geq a>0$ on $(0,\bar R]$. Applying Lemma \ref{Lem2.3}~(ii) on $[0,\bar R]$, we also obtain $\varphi^{\prime\prime}(r)\geq0$ on $(0,\bar R]$. Since no global solution with these properties exists by assumption, $R<+\infty$.
If $\varphi$ remained bounded as $r\to R^-$, the upper estimate \eqref{1prime_n_upper} would also bound $\varphi^\prime$, contradicting \eqref{finite_radius_alternative}. Hence $\lim\limits_{r\to R^-}\varphi(r)=+\infty$.
Since $q<2-\frac{1}{n}$, both exponents $\frac{n-1}{2n-1-qn}$ and $\frac{1}{2n-1-qn}$ are positive and finite. Define
\begin{equation*}
M_R := C_1 \max\big\{ R^{\frac{1}{2n-1}}, R^{\frac{1}{2n-1-qn}} \big\}.
\end{equation*}
Since $f, g \ge 0$, the upper estimate \eqref{1prime_n_upper} gives
\begin{equation*}
\frac{\varphi^{\prime}(r)}{\left(\int_0^{\varphi(r)} f^{\frac{n}{n-1}}(t) d t\right)^{\frac{n-1}{2n-1}} + \left(\int_0^{\varphi(r)} g^{\frac{n}{n-1}}(\tau) d \tau\right)^{\frac{n-1}{2n-1-qn}} } \le M_R, \quad \forall r \in (0, R).
\end{equation*}
Integrating from $0$ to $r$ and using $\varphi(0)=a$, we obtain
\begin{equation*}
\int_{a}^{\varphi(r)} \frac{dt}{ \left(\int_0^t f^{\frac{n}{n-1}}(s) ds\right)^{\frac{n-1}{2n-1}} + \left(\int_0^t g^{\frac{n}{n-1}}(s) ds\right)^{\frac{n-1}{2n-1-qn}} } \le M_R R.
\end{equation*}
Letting $r\to R^-$ gives
$$
\int_{a}^{+\infty} \frac{dt}{ \left(\int_0^t f^{\frac{n}{n-1}}(s) ds\right)^{\frac{n-1}{2n-1}} + \left(\int_0^t g^{\frac{n}{n-1}}(s) ds\right)^{\frac{n-1}{2n-1-qn}} } \le M_R R < +\infty,
$$
which contradicts \eqref{KO_n}.
Thus, $R=+\infty$. Since $\bar R<R$ was arbitrary, this solution satisfies
$$
\varphi\in C^1[0,+\infty)\cap C^2(0,+\infty),\quad \varphi(r)\geq a>0,\quad \varphi^\prime(r)>0,\quad \varphi^{\prime\prime}(r)\geq0,\quad r>0.
$$
Since $a>0$ was arbitrary, the sufficiency is proved.

\underline{\textit{Necessity.}} Suppose that, for some initial value $a>0$, \eqref{CP} admits a positive global solution with the properties stated in the theorem, but \eqref{KO_n} fails. Then there exists $A_0>0$ such that \begin{equation}\label{contra_hyp_n}
\int_{A_0}^{+\infty} \frac{dt}{ \left(\int_0^t f^{\frac{n}{n-1}}(s) ds\right)^{\frac{n-1}{2n-1}} + \left(\int_0^t g^{\frac{n}{n-1}}(s) ds\right)^{\frac{n-1}{2n-1-qn}} } < +\infty.
\end{equation}
Fix $r_0>0$. Since $\varphi^{\prime\prime}\geq0$ and $\varphi^\prime(r_0)>0$,
$$
\varphi(r)\geq\varphi(r_0)+\varphi^\prime(r_0)(r-r_0)\to+\infty\quad\text{as }r\to+\infty.
$$
We may therefore choose $r_1>r_0$ such that $A:=\varphi(r_1)>\max\{a,A_0\}$.
Define
\begin{equation*}
M_{r_1} := C_2 \min\big\{ r_1^{\frac{1}{2n-1}}, r_1^{\frac{1}{2n-1-qn}} \big\} > 0.
\end{equation*}
Integrating \eqref{1prime_n_lower} from $r_1$ to $r$, we deduce
\begin{equation*}
\int_{A}^{\varphi(r)} \frac{dt}{ \left(\int_{a}^t f^{\frac{n}{n-1}}(s) ds\right)^{\frac{n-1}{2n-1}} + \left(\int_{a}^t g^{\frac{n}{n-1}}(s) ds\right)^{\frac{n-1}{2n-1-qn}} } \ge M_{r_1}(r-r_1),\quad r>r_1.
\end{equation*}
For $t\geq2A$, we have $a<A\leq \frac{t}{2}$.
Since $f$ and $g$ are non-decreasing, we have
$$
\int_a^t f^{\frac{n}{n-1}}(s)ds\geq\int_{\frac{t}{2}}^t f^{\frac{n}{n-1}}(s)ds\geq\int_0^{\frac{t}{2}}f^{\frac{n}{n-1}}(s)ds,\quad t\geq 2A,
$$
and
$$
\int_a^t g^{\frac{n}{n-1}}(s)ds\geq\int_{\frac{t}{2}}^t g^{\frac{n}{n-1}}(s)ds\geq\int_0^{\frac{t}{2}}g^{\frac{n}{n-1}}(s)ds,\quad t\geq 2A.
$$
Therefore, splitting the outer integral at $2A$ and using the preceding comparisons, we obtain
\begin{equation}\label{4.10_n}
M_{r_1}(r-r_1)\leq I_2+2\int_{A}^{+\infty}\frac{dt}{ \left(\int_{0}^t f^{\frac{n}{n-1}}(s) ds\right)^{\frac{n-1}{2n-1}}+ \left(\int_{0}^t g^{\frac{n}{n-1}}(s) ds\right)^{\frac{n-1}{2n-1-qn}} },
\end{equation}where $$I_2=\int_{A}^{2A}\frac{dt}{ \left(\int_{a}^t f^{\frac{n}{n-1}}(s) ds\right)^{\frac{n-1}{2n-1}}+ \left(\int_{a}^t g^{\frac{n}{n-1}}(s) ds\right)^{\frac{n-1}{2n-1-qn}} }<+\infty.$$
Since $r$ in \eqref{4.10_n} can be taken arbitrarily large, \eqref{4.10_n} contradicts \eqref{contra_hyp_n}. Thus, the proof of necessity is complete.

The regularity properties in (i) and (ii) follow from Lemma \ref{Lem2.2}.
\end{proof}

\begin{remark}\label{Rem4.local_requirements}
Theorem \ref{Th4.2} extends the local solutions constructed in Theorem \ref{Th3.1} to global solutions under condition \eqref{KO_n}. For $0<q<1$ and $f\equiv0$, this corresponds to the immediately increasing solutions. For $q<0$ and $g(a)=0$, the positivity of the solution follows from  \eqref{Euler_lower_II_zero}. The gradient critical case $q=2-\frac{1}n$ is treated in Proposition \ref{Prop4.critical}.
\end{remark}

\subsection{Global Existence under Condition (B)}

We now allow $g$ to change sign. Replacing $g$ with $g^+$ yields upper derivative estimates of the same form as under condition (A), and the contribution of $g^-$ is incorporated through exponentially weighted Keller--Osserman integrals.

\begin{lemma}\label{Lemma4.2-2}
Assume that $f$ and $g$ satisfy condition (B) and that $0<q\leq2-\frac{1}n$. Let $R\in(0,+\infty]$ and suppose that $\varphi\in C^2[0,R)$ is a solution of \eqref{sODE} satisfying $\varphi(0)=a>0$, $\varphi^\prime(0)=0$, $\varphi^\prime(r)>0$, and $\varphi^{\prime\prime}(r)\geq0$ for all $r\in(0,R)$.
Then, for all $r\in(0,R)$,
\begin{equation}\label{1prime_n_upper_2}
\varphi^{\prime}(r) \leq
\left\{
\begin{alignedat}{2}
&C_1\biggl[r^{\frac{1}{2n-1}}\biggl(\int_a^{\varphi(r)}f^{\frac{n}{n-1}}(t)dt\biggr)^{\theta_1}+r^{\frac{1}{2n-1-qn}}\biggl(\int_a^{\varphi(r)}(g^+(t))^{\frac{n}{n-1}}dt\biggr)^{\theta_2}\biggr],&\quad&0<q<2-\tfrac1n,\\
&\biggl(\Lambda_1r^{\frac{1}{n-1}}\int_a^{\varphi(r)}e^{\Lambda_1r^{\frac{1}{n-1}}\int_t^{\varphi(r)}(g^+(s))^{\frac{n}{n-1}}ds}f^{\frac{n}{n-1}}(t)dt\biggr)^{\theta_1},&\quad&q=2-\tfrac1n.
\end{alignedat}
\right.
\end{equation}
If we further assume $g(a)\geq0$, then for all $r\in(0,R)$,
\begin{equation}\label{1prime_n_lower_2}
\varphi^{\prime}(r) \geq
\left\{
\begin{alignedat}{2}
&C_2\biggl[r^{\frac{1}{2n-1}}\biggl(\int_a^{\varphi(r)}f^{\frac{n}{n-1}}(t)dt\biggr)^{\theta_1}+r^{\frac{1}{2n-1-qn}}\biggl(\int_a^{\varphi(r)}(g^+(t))^{\frac{n}{n-1}}dt\biggr)^{\theta_2}\biggr],&\quad&0<q<2-\tfrac1n,\\
&\biggl(\Gamma_1r^{\frac{1}{n-1}}\int_{\varphi(\frac{r}{2})}^{\varphi(r)}e^{\Gamma_1r^{\frac{1}{n-1}}\int_t^{\varphi(r)}(g^+(s))^{\frac{n}{n-1}}ds}f^{\frac{n}{n-1}}(t)dt\biggr)^{\theta_1},&\quad&q=2-\tfrac1n.
\end{alignedat}
\right.
\end{equation}
Moreover, for $0<q\leq2-\frac{1}n$,
\begin{equation}\label{1prime_n_lower_uni}
\varphi^\prime(r)\geq\left(\widetilde C_1\left(\frac{1}{2}\right)^{2n-2+\frac{n}{n-1}}r^{\frac{1}{n-1}}\int_a^{\varphi(r)}e^{-\Lambda_3r^{\frac{1}{n-1}}\int_t^{\varphi(r)}P_n(\tau)d\tau}f^{\frac{n}{n-1}}(t)dt\right)^{\frac{n-1}{2n-1}},\quad r\in(0,R).
\end{equation}
If we further assume $g(t)\leq0$ for all $t\in\mathbb R$, then, for $0<q\leq2-\frac{1}n$, we have
\begin{equation}\label{1prime_n_upper_uni}
\varphi^\prime(r)\leq\left[\widetilde C_1\int_a^{\varphi(r)}(\varphi^{-1}(t))^{\frac{1}{n-1}}f^{\frac{n}{n-1}}(t)e^{-\widetilde C_1\int_t^{\varphi(r)}(\varphi^{-1}(\tau))^{\frac{1}{n-1}}P_n(\tau)d\tau}dt\right]^{\frac{n-1}{2n-1}},\quad r\in(0,R).
\end{equation}
Here, $\varphi^{-1}$ is the inverse of $\varphi$ on its range, $P_n$ is defined in \eqref{Pn}, and $\widetilde C_1=\frac{2n-1}{n-1}\left(\frac{1}{n-1}\right)^{\frac{1}{n-1}}$. The constant $\Lambda_3>0$ depends only on $n$ and $q$, and $C_1$, $C_2$, $\Gamma_1$, and $\Lambda_1$ are as in Lemma \ref{Lemma4.2}.
\end{lemma}

\begin{proof}
Since $0\leq f(\varphi)+g(\varphi)(\varphi^\prime)^q\leq f(\varphi)+g^+(\varphi)(\varphi^\prime)^q$,
the proof of \eqref{1prime_n_upper_2} is identical to that of \eqref{1prime_n_upper}, with $g$ replaced by $g^+$. If $g(a)\geq0$, then $g(\varphi(r))=g^+(\varphi(r))$ for $r>0$, and the proof of \eqref{1prime_n_lower_2} is identical to that of \eqref{1prime_n_lower}.

We first prove \eqref{1prime_n_lower_uni}. Recall that $\widetilde C_1=\frac{2n-1}{n-1}\left(\frac{1}{n-1}\right)^{\frac{1}{n-1}}$ and $\alpha = \frac{(n-2)(2n-1)}{n-1}$. Let $\psi(r) = (\varphi^{\prime}(r))^{\frac{2n-1}{n-1}}$.
Multiplying \eqref{eq4.exp} by $\frac{2n-1}{n-1}\varphi^\prime(r)$ gives
\begin{equation}\label{4.53}
\psi^\prime(r)+\frac{\alpha}{r}\psi(r)=\widetilde C_1r^{\frac{1}{n-1}}\left[f(\varphi(r))+g(\varphi(r))(\varphi^\prime(r))^q\right]^{\frac{n}{n-1}}\varphi^\prime(r).
\end{equation}
Since $f(\varphi)\leq f(\varphi)+g(\varphi)(\varphi^\prime)^q+g^-(\varphi)(\varphi^\prime)^q$, using the inequality $(A+B)^{\frac{n}{n-1}} \leq 2^{\frac{1}{n-1}} \left(A^{\frac{n}{n-1}}+B^{\frac{n}{n-1}}\right)$ for $A=f(\varphi)+g(\varphi)(\varphi^\prime)^q\geq 0$ and $B=g^-(\varphi)(\varphi^\prime)^q\ge 0$, we obtain
$$
\left[f(\varphi)+g(\varphi)(\varphi^\prime)^q\right]^{\frac{n}{n-1}}\geq2^{-\frac{1}{n-1}}f^{\frac{n}{n-1}}(\varphi)-(g^-(\varphi))^{\frac{n}{n-1}}(\varphi^\prime)^{\frac{qn}{n-1}}.
$$
For $0<q<2-\frac{1}n$, the definition of $P_n$ and Young's inequality yield
$$
(g^-(\varphi))^{\frac{n}{n-1}}(\varphi^\prime)^{\frac{qn}{n-1}}\leq\frac{qn}{2n-1}\left[2^{\frac{n}{n-1}}\frac{2n-1-qn}{2n-1}\right]^{\frac{2n-1-qn}{qn}}P_n(\varphi)\psi+2^{-\frac{n}{n-1}}f^{\frac{n}{n-1}}(\varphi).
$$
When $q=2-\frac{1}n$, the same conclusion holds with $P_n=(g^-)^{\frac{n}{n-1}}$ and with the coefficient in front of $P_n\psi$ equal to $1$. Set
\begin{equation}\label{Lambda_3}
\Lambda_3:=
\begin{cases}
\frac{qn}{n-1}\left(\frac{1}{n-1}\right)^{\frac{1}{n-1}}\left(2^{\frac{n}{n-1}}\frac{2n-1-qn}{2n-1}\right)^{\frac{2n-1-qn}{qn}},&0<q<2-\frac{1}n,\\
\frac{2n-1}{n-1}\left(\frac{1}{n-1}\right)^{\frac{1}{n-1}},&q=2-\frac{1}n.
\end{cases}
\end{equation}
Substituting into \eqref{4.53}, we obtain
$$
\psi^\prime(r)+\frac{\alpha}{r}\psi(r)+\Lambda_3r^{\frac{1}{n-1}}P_n(\varphi(r))\psi(r)\varphi^\prime(r)\geq\widetilde C_1 2^{-\frac{n}{n-1}}r^{\frac{1}{n-1}}f^{\frac{n}{n-1}}(\varphi(r))\varphi^\prime(r).
$$
Multiplying by $r^\alpha e^{\Lambda_3\int_0^r \tau^{\frac{1}{n-1}} P_n(\varphi(\tau)) \varphi^\prime(\tau) d\tau}$, we obtain
\begin{align*}
\left( r^\alpha e^{\Lambda_3\int_0^r \tau^{\frac{1}{n-1}} P_n(\varphi(\tau)) \varphi^\prime(\tau) d\tau} \psi(r) \right)^\prime
\ge \widetilde C_1 2^{-\frac{n}{n-1}}r^{2n-3} e^{\Lambda_3\int_0^r\tau^{\frac{1}{n-1}}P_n(\varphi(\tau))\varphi^\prime(\tau)d\tau}f^{\frac{n}{n-1}}(\varphi(r))\varphi^\prime(r).
\end{align*}
Integrating over $[\frac{r}{2},r]$, and using $\tau^{\frac{1}{n-1}}\leq r^{\frac{1}{n-1}}$ for $\tau\in[\frac{r}{2},r]$, we obtain
\begin{equation}\label{4.55}
\psi(r)\geq\widetilde C_1\left(\frac{1}{2}\right)^{2n-3+\frac{n}{n-1}}r^{\frac{1}{n-1}}\int_{\varphi(\frac{r}{2})}^{\varphi(r)}e^{-\Lambda_3r^{\frac{1}{n-1}}\int_t^{\varphi(r)}P_n(\tau)d\tau}f^{\frac{n}{n-1}}(t)dt.
\end{equation}
For any fixed $r>0$, define
$$
\mathcal{I}_n(t):=f^{\frac{n}{n-1}}(t)e^{-\Lambda_3r^{\frac{1}{n-1}}\int_t^{\varphi(r)}P_n(\tau)d\tau}, \quad a\leq t\leq\varphi(r).
$$
Since $f$ is non-decreasing and $P_n\geq0$, the function $\mathcal{I}_n$ is non-decreasing in $t$.
Combined with the convexity of $\varphi$, we obtain
\begin{equation*}
\int_a^{\varphi(r)} \mathcal{I}_n(t) dt = \int_a^{\varphi(\frac{r}{2})} \mathcal{I}_n(t) dt + \int_{\varphi(\frac{r}{2})}^{\varphi(r)} \mathcal{I}_n(t) dt \le 2 \int_{\varphi(\frac{r}{2})}^{\varphi(r)} \mathcal{I}_n(t) dt,\quad \forall r>0.
\end{equation*}
Substituting this into \eqref{4.55} and taking the power $\frac{n-1}{2n-1}$ proves \eqref{1prime_n_lower_uni}.

Next, we prove \eqref{1prime_n_upper_uni}. Suppose now that $g\leq0$. Then $g=-g^-$. From Lemma~\ref{Lem2.3}~(iii), we have $$
0\leq\frac{g^-(\varphi(r))(\varphi^\prime(r))^q}{f(\varphi(r))}\leq1.
$$
Using $(1-A)^{\frac{n}{n-1}}\leq1-A$ for $A\in[0,1]$, we obtain
$$
\left[f(\varphi)-g^-(\varphi)(\varphi^\prime)^q\right]^{\frac{n}{n-1}}\leq f^{\frac{n}{n-1}}(\varphi)-f^{\frac{1}{n-1}}(\varphi)g^-(\varphi)(\varphi^\prime)^q.
$$
Moreover, $\frac{2n-1}{n-1}-q>0$ for $0<q\leq2-\frac{1}n$. Hence the slope bound in Lemma \ref{Lem2.3}~(iii) implies
\begin{align*}
f^{\frac{1}{n-1}}(\varphi(r))g^-(\varphi(r))(\varphi^\prime(r))^q
&=f^{\frac{1}{n-1}}(\varphi(r))g^-(\varphi(r))\frac{(\varphi^\prime(r))^{\frac{2n-1}{n-1}}}{(\varphi^\prime(r))^{\frac{2n-1-q(n-1)}{n-1}}}\\
&\geq f^{\frac{1}{n-1}}(\varphi(r))g^-(\varphi(r))\left(\frac{g^-(\varphi(r))}{f(\varphi(r))}\right)^{\frac{2n-1-q(n-1)}{q(n-1)}}(\varphi^\prime(r))^{\frac{2n-1}{n-1}}\\
&=P_n(\varphi(r))(\varphi^\prime(r))^{\frac{2n-1}{n-1}}.
\end{align*}
Substituting these estimates into \eqref{4.53} and dropping the non-negative term $\frac{\alpha}{r}\psi$, we obtain
$$
\psi^\prime(r)+\widetilde C_1r^{\frac{1}{n-1}}P_n(\varphi(r))\psi(r)\varphi^\prime(r)\leq\widetilde C_1r^{\frac{1}{n-1}}f^{\frac{n}{n-1}}(\varphi(r))\varphi^\prime(r).
$$
Multiplying both sides by $e^{\widetilde C_1\int_0^r\tau^{\frac{1}{n-1}}P_n(\varphi(\tau))\varphi^\prime(\tau)d\tau}$ and integrating from $0$ to $r$ gives
$$
\psi(r)\leq\widetilde C_1\int_0^rs^{\frac{1}{n-1}}e^{-\widetilde C_1\int_s^r\tau^{\frac{1}{n-1}}P_n(\varphi(\tau))\varphi^\prime(\tau)d\tau}f^{\frac{n}{n-1}}(\varphi(s))\varphi^\prime(s)ds.
$$
Since $\varphi^\prime(r)>0$ for $r>0$, the function $\varphi$ is strictly increasing, and hence its inverse $\varphi^{-1}$ is well defined on the range of $\varphi$.
Thus, \begin{align*}
\psi(r)
\leq\widetilde C_1\int_a^{\varphi(r)}\left(\varphi^{-1}(t)\right)^{\frac{1}{n-1}}f^{\frac{n}{n-1}}(t)e^{-\widetilde C_1\int_t^{\varphi(r)}\left(\varphi^{-1}(\tau)\right)^{\frac{1}{n-1}}P_n(\tau)d\tau}
dt.
\end{align*}
Substituting back $\psi(r)=(\varphi^\prime(r))^{\frac{2n-1}{n-1}}$ and taking the power $\frac{n-1}{2n-1}$ on both sides yields \eqref{1prime_n_upper_uni}.
\end{proof}

\begin{theorem}[Global Existence under Condition (B)]\label{Th4.2-2}
Assume that $f$ and $g$ satisfy condition (B).
Then the following conclusions hold.
\begin{itemize}
\item[(i)] Let $0<q<2-\frac{1}n$ and assume that $\lim\limits_{t\to+\infty}g(t)>0$. For any initial value $a>0$, the Cauchy problem \eqref{CP} admits a positive global solution $\varphi \in C^2[0,+\infty)$ satisfying $\varphi^\prime(r)>0$ and $\varphi^{\prime\prime}(r)\geq0$ for all $r>0$
if and only if condition \eqref{KO_n_2} holds.

\item[(ii)] Let $0<q\leq2-\frac{1}n$ and assume that $\lim\limits_{t\to+\infty}g(t)\leq0$. For any initial value $a>0$, the Cauchy problem \eqref{CP} admits a positive global solution $\varphi\in C^2[0,+\infty)$ satisfying $\varphi^\prime(r)>0$ and $\varphi^{\prime\prime}(r)\geq0$ for all $r>0$ if and only if condition \eqref{KO_n_3} holds.
\end{itemize}
\end{theorem}

\begin{proof}
The regularity, positivity, monotonicity, and convexity of solutions follow from Theorem \ref{Th3.2} and Lemmas \ref{Lem2.2}--\ref{Lem2.3}. Moreover, the corresponding upper estimates in Lemma \ref{Lemma4.2-2} show that if the maximal radius $R<+\infty$ and $\varphi$ remains bounded as $r\to R^-$, then $\varphi^\prime$ remains bounded as well, contradicting \eqref{finite_radius_alternative}. Consequently, if $R<+\infty$, then $\lim\limits_{r\to R^-}\varphi(r)=+\infty$.

We first prove (i). The sufficiency follows from the proof of Theorem \ref{Th4.2} using the upper estimate \eqref{1prime_n_upper_2}. Conversely, assume that such a global solution exists. Since $\varphi(r)\to+\infty$ as $r\to+\infty$, choose $r_0>0$ such that $g(\varphi(r))>0$ for $r\geq r_0$, and set $A:=\varphi(r_0)$. For $r\geq2r_0$, we have $g(\varphi(s))=g^+(\varphi(s))$ for $s\in[r/2,r]$. Repeating the integrations leading to \eqref{lower_f_2} and \eqref{lower_g_2}, and using the convexity of $\varphi$ and the monotonicity of $f$ and $g^+$, we obtain, for some $c_0>0$,
$$
\varphi^\prime(r)\geq c_0\left[\left(\int_A^{\varphi(r)}f^{\frac{n}{n-1}}(t)dt\right)^{\theta_1}+\left(\int_A^{\varphi(r)}(g^+(t))^{\frac{n}{n-1}}dt\right)^{\theta_2}\right],\quad r\geq2r_0.
$$
The same integration and comparison of finite lower limits as in the necessity proof of Theorem \ref{Th4.2} then yield \eqref{KO_n_2}. This proves (i).

We next prove (ii). The monotonicity of $g$ and $\lim\limits_{t\to+\infty}g(t)\leq0$ imply $g\leq0$ on $\mathbb R$. Thus $g^-=-g$ is nonnegative and non-increasing. Since $f$ is positive and non-decreasing on $(0,+\infty)$ and $q\leq2-\frac{1}n$, \eqref{Pn} shows that $P_n$ is also nonnegative and non-increasing. We use $J_{\lambda,b}^-$ from \eqref{weighted_J_definition} with $F=f^{\frac{n}{n-1}}$ and $P=P_n$.

\underline{\textit{Sufficiency.}} Assume that \eqref{KO_n_3} holds and, by contradiction, that the maximal radius $R$ is finite. Fix $r_0\in(0,R)$,  set  $A:=\varphi(r_0)$ and  $\lambda_1:=\widetilde C_1r_0^{\frac{1}{n-1}}>0$.
For $a\leq t\leq A$, we have $\varphi^{-1}(t)\leq r_0$, and the exponential factor in \eqref{1prime_n_upper_uni} does not exceed $1$. Hence
$$
\widetilde C_1\int_a^A(\varphi^{-1}(t))^{\frac{1}{n-1}}f^{\frac{n}{n-1}}(t)e^{-\widetilde C_1\int_t^{\varphi(r)}(\varphi^{-1}(\tau))^{\frac{1}{n-1}}P_n(\tau)d\tau}dt\leq K_0,
$$
where $K_0:=\widetilde C_1r_0^{\frac{1}{n-1}}\int_a^A f^{\frac{n}{n-1}}(t)dt$.
For $A\leq t\leq\varphi(r)$, we have $\varphi^{-1}(t)\leq R$ and $\varphi^{-1}(\tau)\geq r_0$ for $t\leq\tau\leq\varphi(r)$. Since $P_n\geq0$,
$$
\widetilde C_1\int_A^{\varphi(r)}(\varphi^{-1}(t))^{\frac{1}{n-1}}f^{\frac{n}{n-1}}(t)e^{-\widetilde C_1\int_t^{\varphi(r)}(\varphi^{-1}(\tau))^{\frac{1}{n-1}}P_n(\tau)d\tau}dt\leq\widetilde C_1R^{\frac{1}{n-1}}J_{\lambda_1,A}^-(\varphi(r)).
$$
Raising \eqref{1prime_n_upper_uni} to the power $\frac{2n-1}{n-1}$ and combining the preceding two estimates, we obtain
\begin{equation}\label{weighted_upper_with_constant}
(\varphi^\prime(r))^{\frac{2n-1}{n-1}}\leq K_0+\widetilde C_1R^{\frac{1}{n-1}}J_{\lambda_1,A}^-(\varphi(r)),\quad r_0\leq r<R.
\end{equation}
By Lemma \ref{Lem2.4}~(i), $J_{\lambda_1,A}^-(t)$ is positive and non-decreasing for $t>A$.
Fix $T>A$.
For $t\geq T$, the monotonicity of $J_{\lambda_1,A}^-$ gives $J_{\lambda_1,A}^-(t)\geq J_{\lambda_1,A}^-(T)$, and hence
$$
K_0\leq\frac{K_0}{J_{\lambda_1,A}^-(T)}J_{\lambda_1,A}^-(t).
$$
Set
$
K_1:=\frac{K_0}{J_{\lambda_1,A}^-(T)}+\widetilde C_1R^{\frac{1}{n-1}}.
$
Since $\varphi(r)\to+\infty$ as $r\to R^-$, there exists $r_T\in(r_0,R)$ such that $\varphi(r_T)=T$. It follows from \eqref{weighted_upper_with_constant} that
$$
\varphi^\prime(r)\leq K_1^{\frac{n-1}{2n-1}}\left(J_{\lambda_1,A}^-(\varphi(r))\right)^{\frac{n-1}{2n-1}},\quad r_T\leq r<R.
$$
Integrating from $r_T$ to $r$, we obtain
$$
\int_T^{\varphi(r)}\frac{dt}{\left(J_{\lambda_1,A}^-(t)\right)^{\frac{n-1}{2n-1}}}\leq K_1^{\frac{n-1}{2n-1}}(r-r_T)\leq K_1^{\frac{n-1}{2n-1}}R<+\infty.
$$
Moreover, $J_{\lambda_1,0}^-(t)\geq J_{\lambda_1,A}^-(t)$ for $t\geq T$. Letting $r\to R^-$ gives
$$
\int_T^{+\infty}\frac{dt}{\left(J_{\lambda_1,0}^-(t)\right)^{\frac{n-1}{2n-1}}}<+\infty.
$$
By \eqref{weighted_parameter_comparison}, it follows that
$$
\int_T^{+\infty}\frac{dt}{\left(J_{1,0}^-(t)\right)^{\frac{n-1}{2n-1}}}<+\infty,
$$
which contradicts \eqref{KO_n_3}. Therefore, $R=+\infty$.

\underline{\textit{Necessity.}} Assume that \eqref{CP} admits a positive global solution for some initial value $a>0$ and that \eqref{KO_n_3} fails.
Since \eqref{KO_n_3} fails, we may take $\lambda_0=1$ and choose $A_0>0$ such that
$$
\int_{A_0}^{+\infty}\frac{dt}{(J_{\lambda_0,0}^-(t))^{\frac{n-1}{2n-1}}}<+\infty.
$$
Since $\varphi$ is convex and $\varphi^\prime>0$, we have
$\varphi(r)\to+\infty$ as $r\to+\infty$.
Hence, we may choose $r_1>0$ such that $$
A:=\varphi(r_1)>\max\{a,A_0\},\quad \Lambda_3r_1^{\frac{1}{n-1}}\geq\lambda_0.
$$
For $r\geq r_1$, applying \eqref{weighted_parameter_comparison} with $\mu=\lambda_0$ and $\lambda=\Lambda_3r^{\frac{1}{n-1}}$, it gives
$$
r^{\frac{1}{n-1}}J_{\Lambda_3r^{\frac{1}{n-1}},a}^-(t)\geq\frac{\lambda_0}{\Lambda_3}J_{\lambda_0,a}^-(t),\quad t>a.
$$
Combining this inequality with \eqref{1prime_n_lower_uni}, we obtain
\begin{equation}\label{necessity_lower_J}
\varphi^\prime(r)\geq C\left(J_{\lambda_0,a}^-(\varphi(r))\right)^{\frac{n-1}{2n-1}},\quad r\geq r_1,
\end{equation}
where $C=\left[\widetilde C_1\left(\frac{1}{2}\right)^{2n-2+\frac{n}{n-1}}\frac{\lambda_0}{\Lambda_3}\right]^{\frac{n-1}{2n-1}}>0$.
Integrating \eqref{necessity_lower_J} from $r_1$ to $r$, we obtain
\begin{equation}\label{necessity_integral_J}
C(r-r_1)\leq\int_A^{\varphi(r)}\frac{dt}{\left(J_{\lambda_0,a}^-(t)\right)^{\frac{n-1}{2n-1}}}\leq\int_A^{+\infty}\frac{dt}{\left(J_{\lambda_0,a}^-(t)\right)^{\frac{n-1}{2n-1}}}.
\end{equation}
We claim that the last integral is finite.
Since $A>a$ and $f>0$ on $(0,+\infty)$, we have $J_{\lambda_0,a}^-(A)>0$. The continuity of $J_{\lambda_0,a}^-$ gives
$$
\int_A^{2A}\frac{dt}{\left(J_{\lambda_0,a}^-(t)\right)^{\frac{n-1}{2n-1}}}<+\infty.
$$
For $t\geq2A$, we have $a<A\leq\frac{t}{2}$. Moreover, for $0\leq s\leq\frac{t}{2}$, the monotonicity of $f$ and $P_n$ gives
\begin{align*}
J_{\lambda_0,a}^-(t)\geq\int_{\frac{t}{2}}^t f^{\frac{n}{n-1}}(s)e^{-\lambda_0\int_s^tP_n(\tau)d\tau}ds
\geq\int_0^{\frac{t}{2}}f^{\frac{n}{n-1}}(s)e^{-\lambda_0\int_s^{\frac{t}{2}}P_n(\tau)d\tau}ds=J_{\lambda_0,0}^-\left(\frac{t}{2}\right).
\end{align*}
Therefore,
\begin{align*}
\int_{2A}^{+\infty}\frac{dt}{\left(J_{\lambda_0,a}^-(t)\right)^{\frac{n-1}{2n-1}}}
\leq\int_{2A}^{+\infty}\frac{dt}{\left(J_{\lambda_0,0}^-\left(\frac{t}2\right)\right)^{\frac{n-1}{2n-1}}}
=2\int_A^{+\infty}\frac{ds}{\left(J_{\lambda_0,0}^-(s)\right)^{\frac{n-1}{2n-1}}}<+\infty,
\end{align*}
where the last inequality follows from $A>A_0$. Combining the estimates over $[A,2A]$ and $[2A,+\infty)$, we conclude that
$$
\int_A^{+\infty}\frac{dt}{\left(J_{\lambda_0,a}^-(t)\right)^{\frac{n-1}{2n-1}}}<+\infty.
$$
Thus, the right-hand side of \eqref{necessity_integral_J} is bounded independently of $r$, whereas its left-hand side tends to $+\infty$ as $r\to+\infty$. This contradiction proves \eqref{KO_n_3} and completes the proof.
\end{proof}

We conclude this section with the case $q=2-\frac{1}n$. When the gradient term vanishes, the equation is independent of $q$ and is governed by \eqref{KO_n_critical_zero}. Eventual positivity of the gradient coefficient instead yields finite-radius blow-up.

\begin{proposition}\label{Prop4.critical}
Let $q=2-\frac{1}n$. Then the following assertions hold.
\begin{enumerate}
\item[(i)] Assume that condition (A) or (B) holds and $g\equiv0$ on $(0,+\infty)$. For any initial value $a>0$, problem \eqref{CP} admits a positive global solution $\varphi\in C^2[0,+\infty)$ satisfying $\varphi^\prime(r)>0$ and $\varphi^{\prime\prime}(r)\geq0$ for all $r>0$ if and only if \eqref{KO_n_critical_zero} holds.

\item[(ii)] Assume either that condition (A) holds, $f(t)>0$ for all $t>0$, and $g\not\equiv0$ on $(0,+\infty)$, or that condition (B) holds and $\lim\limits_{t\to+\infty}g(t)>0$. Then, for any initial value $a>0$, problem \eqref{CP} admits no positive global solution. Moreover, it admits a positive solution $\varphi\in C^2[0,R)$ with a finite maximal radius $R>0$ such that
$\varphi^\prime(r)>0$, $\varphi^{\prime\prime}(r)\geq0$ for $r\in(0,R)$, and $\lim\limits_{r\to R^-}\varphi(r)=+\infty$.
\end{enumerate}
\end{proposition}

\begin{proof}
We first prove (i). In either alternative, $f(t)>0$ for $t>0$. By Theorem \ref{Th3.1} or \ref{Th3.2}, \eqref{CP} admits a positive local solution. Choose a maximal continuation on $[0,R)$. Since the equation is independent of $q$ when $g=0$, the estimates in Lemma \ref{Lemma4.2} with $q=0$ give constants $c,C>0$ depending only on $n$ such that
$$
c r^{\frac{1}{2n-1}}\left(\int_a^{\varphi(r)}f^{\frac{n}{n-1}}(s)ds\right)^{\theta_1}\leq\varphi^\prime(r)\leq C r^{\frac{1}{2n-1}}\left(\int_a^{\varphi(r)}f^{\frac{n}{n-1}}(s)ds\right)^{\theta_1},\quad r\in(0,R).
$$
Separating variables and integrating from $0$ to $r$, we obtain
\begin{equation}\label{4.25}
\frac{(2n-1)c}{2n}r^{\frac{2n}{2n-1}}\leq\int_a^{\varphi(r)}\frac{dt}{\left(\int_a^t f^{\frac{n}{n-1}}(s)ds\right)^{\theta_1}}\leq\frac{(2n-1)C}{2n}r^{\frac{2n}{2n-1}},\quad r\in(0,R).
\end{equation}
The middle integral in \eqref{4.25} is finite near $t=a$, since $f(a)>0$ and $\theta_1<1$. Moreover, $\int_a^t f^{\frac{n}{n-1}}(s)ds\to+\infty$ as $t\to+\infty$. For all sufficiently large $t$,
$$
\int_a^t f^{\frac{n}{n-1}}(s)ds\leq\int_0^t f^{\frac{n}{n-1}}(s)ds\leq2\int_a^t f^{\frac{n}{n-1}}(s)ds.
$$
Thus the divergence of the middle integral at infinity is equivalent to \eqref{KO_n_critical_zero}.

Assume that \eqref{KO_n_critical_zero} holds. If $R<+\infty$ and $\varphi$ remained bounded as $r\to R^-$, the upper derivative estimate would bound $\varphi^\prime$, contradicting \eqref{finite_radius_alternative}. Hence $\varphi(r)\to+\infty$ as $r\to R^-$. Letting $r\to R^-$ in the upper bound from \eqref{4.25} then contradicts \eqref{KO_n_critical_zero}. Therefore $R=+\infty$. Conversely, if a positive global solution exists, convexity together with $\varphi^\prime>0$ imply $\varphi(r)\to+\infty$ as $r\to+\infty$. The lower bound from \eqref{4.25} then implies \eqref{KO_n_critical_zero}.

We next prove (ii). It follows from Theorem \ref{Th3.1} or \ref{Th3.2} that \eqref{CP} admits a positive local solution for any $a>0$. Let $\varphi$ be any maximal continuation of such a solution on $[0,R)$. By Lemma \ref{Lem2.3}, we have $\varphi^\prime>0$ and $\varphi^{\prime\prime}\geq0$ on $(0,R)$.

Suppose that $R=+\infty$. Since $\varphi^\prime$ is positive and non-decreasing, $\varphi(r)\to+\infty$ as $r\to+\infty$. In either alternative, $g$ is positive for all sufficiently large arguments. Hence there exist $r_0,m>0$ such that $g(\varphi(r))\geq m$ for $r\geq r_0$. Since $f\geq0$ and $q=2-\frac{1}n$, equation \eqref{sODE} gives
$$
\varphi^{\prime\prime}(r)+\frac{n-2}{r}\varphi^\prime(r)\geq(n-1)^{-\frac{1}{n-1}}m^{\frac{n}{n-1}}r^{\frac{1}{n-1}}(\varphi^\prime(r))^2,\quad r\geq r_0.
$$
Let $K:=(n-1)^{-\frac{1}{n-1}}m^{\frac{n}{n-1}}>0$. Since $\varphi^\prime(r)\geq \varphi^\prime(r_0)>0$, we may choose $r_1\geq r_0$ such that
$$
\frac{n-2}{r}\varphi^\prime(r)\leq\frac{K}{2}r^{\frac{1}{n-1}}(\varphi^\prime(r))^2,\quad r\geq r_1.
$$
Consequently,
$$
\varphi^{\prime\prime}(r)\geq\frac{K}{2}r^{\frac{1}{n-1}}(\varphi^\prime(r))^2,\quad r\geq r_1.
$$
Integrating from $r_1$ to $r$ yields
$$
\frac{1}{\varphi^\prime(r)}\leq\frac{1}{\varphi^\prime(r_1)}-\frac{K}{2}\int_{r_1}^r s^{\frac{1}{n-1}}ds.
$$
The right-hand side becomes negative for large $r$, a contradiction. Thus $R<+\infty$.

Finally, if $\varphi$ remained bounded as $r\to R^-$, the upper estimate \eqref{1prime_n_upper} under condition (A), or \eqref{1prime_n_upper_2} under condition (B), would bound $\varphi^\prime$ on $(0,R)$. This contradicts \eqref{finite_radius_alternative}. Hence $\varphi(r)\to+\infty$ as $r\to R^-$, completing the proof of (ii).
\end{proof}

\begin{remark}\label{Rem4.critical_exponents}
Let $1<q<2-\frac{1}n$. Suppose that $f$ and $g$ satisfy condition (A), $f(t)>0$ for $t>0$, and $g\not\equiv0$ on $(0,+\infty)$. Choose $t_0>0$ such that $g(t_0)>0$. Since $g$ is non-decreasing,
$$
\int_0^t g^{\frac{n}{n-1}}(s)ds\geq g^{\frac{n}{n-1}}(t_0)(t-t_0)\geq\frac{1}{2} g^{\frac{n}{n-1}}(t_0)t,\quad t\geq2t_0.
$$
Since $\frac{n-1}{2n-1-qn}>1$, the integrand in \eqref{KO_n} is bounded above by a constant multiple of $t^{-\frac{n-1}{2n-1-qn}}$. Hence the integral in \eqref{KO_n} converges, and the required divergence condition fails.
Under condition (B) with $\lim\limits_{t\to+\infty}g(t)>0$, the same estimate applied to $g^+$ shows that \eqref{KO_n_2} fails.

For $1<q<2-\frac{1}n$, Theorems \ref{Th4.2} and \ref{Th4.2-2}, together with Remark \ref{Rem3.1}, yield increasing convex radial solutions that blow up at a finite radius. Theorems \ref{Th1.1} and \ref{Th1.2} give the corresponding nonexistence of positive nontrivial $(n-1)$-convex entire subsolutions.
For $q=2-\frac{1}n$, Proposition \ref{Prop4.critical}, Theorem \ref{Th1.1}~(iii), and Theorem \ref{Th1.2}~(iii) give the same conclusions under the corresponding hypotheses. Thus, in the range considered here, global existence requires $q\leq1$ when $g$ is eventually positive. Under condition (B) with $\lim\limits_{t\to+\infty}g(t)\leq0$, global existence is instead governed by \eqref{KO_n_3}.
\end{remark}

\section{Proofs of the Main Theorems}\label{Section5}

In this section, we compare an entire subsolution with a radial solution that blows up at a finite radius.
This local comparison gives the equivalence theorem, which combines with Theorems \ref{Th4.2}--\ref{Th4.2-2} and Proposition \ref{Prop4.critical} to prove the main results.

\subsection{Comparison and Equivalence Principles}

We begin with the comparison principle used in the equivalence argument.

\begin{lemma}[Local comparison principle]\label{Lem5.1}
Let $0<R<+\infty$, let $u$ be a positive $(n-1)$-convex entire subsolution of \eqref{eq1.1}, and let $\varphi\in C^1[0,R)\cap C^2(0,R)$ be a positive solution of \eqref{CP} satisfying $\varphi^\prime(r)>0$, $\varphi^{\prime\prime}(r)\geq0$ for $r\in(0,R)$, and $\lim\limits_{r\to R^-}\varphi(r)=+\infty$.
Set $v(x):=\varphi(|x|)$ in $B_R$. Assume that one of the following alternatives holds.
\begin{enumerate}[label=(\roman*)]
\item $q=0$, $u\in C^2(\mathbb R^n)$, and condition (A) holds;
\item $q>0$, $u\in C^2(\mathbb R^n)$, condition (A) or (B) holds, and $f(t)>0$ for all $t>0$;
\item $q<0$, condition (A) holds, and $u$ is a strong subsolution in the sense of Definition \ref{def:entire_(n-1)-convex_subsolutions};
\item $0<q<1$, $u\in C^2(\mathbb R^n)$, condition (A) holds, and $Du(0)\neq0$.
\end{enumerate}
Then
\begin{equation}\label{eq5.1}
u(x)\leq v(x),\quad x\in B_R.
\end{equation}
\end{lemma}

\begin{proof}
Let $\mathcal L[w]:=\operatorname{det}^{\frac{1}n}\bigl(\Delta wI-D^2w\bigr)$.
In every case, $u$ is continuous in $\mathbb R^n$ and bounded on $\overline{B_R}$. Moreover, $v$ is continuous in $B_R$, $Dv(0)=0$, and
$$
\lim\limits_{|x|\to R^-}v(x)=+\infty.
$$
Suppose, by contradiction, that $u>v$ at some point of $B_R$. Then $u-v$ attains a positive maximum in $B_R$. Set $m:=\max\limits_{B_R}(u-v)>0$ and
$$
\mathcal M:=\left\{x\in B_R:(u-v)(x)=m\right\}.
$$

\textit{Case 1: $\mathcal M\cap(B_R\setminus\{0\})\neq\varnothing$.} Fix $x_0\in\mathcal M\cap(B_R\setminus\{0\})$. Since $u-v$ attains its maximum at $x_0$,
$$
Du(x_0)=Dv(x_0)=:p_0.
$$
The assumptions $x_0\neq0$ and $\varphi^\prime>0$ give
$$
|p_0|=|Dv(x_0)|=\varphi^\prime(|x_0|)>0.
$$
In the singular case, $u\in C^2$ and $|Du|>0$ in a neighborhood of $x_0$, and the subsolution inequality holds pointwise.
By Lemmas \ref{Lem2.1}--\ref{Lem2.3} and Remark \ref{Rem2.2}, the radial function $v$ is strictly $(n-1)$-convex in $B_R\setminus\{0\}$ under the corresponding assumptions. Since $u$ is a subsolution, $v$ is a solution, $u(x_0)>v(x_0)$, and $f$ and $g$ are non-decreasing, we obtain
\begin{align*}
\mathcal L[u](x_0)&\geq f(u(x_0))+g(u(x_0))|Du(x_0)|^q\\
&\geq f(v(x_0))+g(v(x_0))|Dv(x_0)|^q\\
&=\mathcal L[v](x_0)>0.
\end{align*}
Thus, both $u$ and $v$ are strictly $(n-1)$-convex at $x_0$.
There exists $\rho_1>0$ such that $\overline{B_{\rho_1}(x_0)}\subset B_R\setminus\{0\}$, $u>v$ in $\overline{B_{\rho_1}(x_0)}$, and both functions are strictly $(n-1)$-convex there. For $\theta\in[0,1]$, set
$$
u_\theta:=\theta u+(1-\theta)v.
$$
Then each $u_\theta$ is strictly $(n-1)$-convex in $\overline{B_{\rho_1}(x_0)}$. By the continuity of $Du$ and $Dv$, there exists $\rho_0\in(0,\rho_1]$ such that
$$
|Du(x)-p_0|\leq\frac{|p_0|}{4},\quad |Dv(x)-p_0|\leq\frac{|p_0|}{4},\quad x\in\overline{B_{\rho_0}(x_0)}.
$$
Consequently, for $x\in\overline{B_{\rho_0}(x_0)}$ and $\theta\in[0,1]$,
\begin{equation}\label{interpolating_gradient}
|Du_\theta(x)|\geq |p_0|-\theta|Du(x)-p_0|-(1-\theta)|Dv(x)-p_0|\geq\frac{|p_0|}{2}>0.
\end{equation}

Let $w:=u-v$. By the fundamental theorem of calculus,
$$
\mathcal L[u]-\mathcal L[v]=\sum_{i,j=1}^na^{ij}D_{ij}w,
$$
where
$$
a^{ij}(x):= \frac{1}n\int_0^1
\det(M_\theta
)^{\frac{1}n}\left[\operatorname{tr}( M_\theta^{-1}) \delta_{ij}-(M_\theta^{-1})_{ij}\right]d\theta,\quad M_\theta:=\Delta u_\theta I-D^2u_\theta.
$$
The eigenvalues of $\operatorname{tr}(M_\theta^{-1})I-M_\theta^{-1}$ are sums of $n-1$ positive eigenvalues of $M_\theta^{-1}$. By compactness of $\overline{B_{\rho_0}(x_0)}\times[0,1]$, $(a^{ij})$ is uniformly positive definite on this ball. If $q\neq0$, then
$$
|Du|^q-|Dv|^q=\sum_{i=1}^nb^iD_iw,
$$
where
$$
b^i(x):=q\int_0^1|Du_\theta|^{q-2}D_iu_\theta d\theta.
$$
The coefficients $b^i$ are bounded in $B_{\rho_0}(x_0)$ by \eqref{interpolating_gradient}. When $q=0$, set $b^i:=0$.
Since $u$ is a subsolution and $v$ is a solution, the monotonicity of $f$ and $g$ gives
\begin{align*}
\sum_{i,j=1}^na^{ij}D_{ij}w-g(u)\sum_{i=1}^nb^iD_iw&\geq f(u)-f(v)+\bigl(g(u)-g(v)\bigr)|Dv|^q\\
&\geq0\quad\text{in }B_{\rho_0}(x_0).
\end{align*}
Since $w$ attains its maximum at $x_0$, the strong maximum principle \cite[Theorem 3.5]{GT1988} yields
$$
w\equiv m\quad\text{in }B_{\rho_0}(x_0).
$$
Therefore, $B_{\rho_0}(x_0)\subset\mathcal M$, and hence $\mathcal M\cap(B_R\setminus\{0\})$ is open in $B_R\setminus\{0\}$.
Since $u-v$ is continuous, $\mathcal M\cap(B_R\setminus\{0\})$ is relatively closed in $B_R\setminus\{0\}$.
Since $B_R\setminus\{0\}$ is connected for $n\geq2$,
$$
\mathcal M\cap(B_R\setminus\{0\})=B_R\setminus\{0\}.
$$
This contradicts
$$
\lim\limits_{|x|\to R^-}\bigl(u(x)-v(x)\bigr)=-\infty.
$$
Hence, a positive maximum of $u-v$ cannot be attained in $B_R\setminus\{0\}$.

\textit{Case 2: $\mathcal M=\{0\}$.} Since $u-v$ attains a local maximum at the origin and $Dv(0)=0$, we have
$$
Du(0)=Dv(0)=0.
$$
This immediately contradicts the assumption $Du(0)\neq0$ in case (iv).
Hence, it remains to consider cases (i)--(iii).
Suppose first that $q=0$ or $q\geq1$. Lemma \ref{Lem2.2}~(i)--(ii) and Lemma \ref{Lem2.1} give $v\in C^2(B_R)$ in these cases. If $q=0$, condition (A) gives
$$
\mathcal L[v](0)=f(v(0))+g(v(0))>0.
$$
If $q\geq1$, the assumption $f(t)>0$ for all $t>0$ gives
$$
\mathcal L[v](0)=f(v(0))>0.
$$
Since $u(0)>v(0)$, the subsolution inequality and the monotonicity of $f$ and $g$ imply that
$$\mathcal L[u](0)\geq\mathcal L[v](0)>0.$$
By continuity, there exists $\rho_*\in(0,R)$ such that $u>v$ and both functions are strictly $(n-1)$-convex in $B_{\rho_*}$.
When $q=0$, the gradient difference is absent. When $q\geq1$, the map $p\mapsto|p|^q$ is locally Lipschitz.
The same second-order linearization and the strong maximum principle give
$$
u-v\equiv m\quad\text{in }B_{\rho_*},
$$
which contradicts $\mathcal M=\{0\}$.

It remains to consider $0<q<1$ in case (ii) and $q<0$ in case (iii). In both cases, $u,v\in C^1(B_R)$ and $Du(0)=Dv(0)=0$. When $q<0$, the definition of a strong subsolution gives $|Du|>0$ almost everywhere in $\mathbb R^n\setminus\{0\}$. When $0<q<1$, if $Du$ vanished in a neighborhood of the origin, then $u$ would be constant there, and the subsolution inequality would give
$$
0=\mathcal L[u]\geq f(u)>0,
$$
a contradiction. Therefore, in either case, we may choose $\xi\in\mathbb R^n\setminus\{0\}$ sufficiently close to the origin such that $Du(\xi)\neq0$ and $u(0)-\varphi(|\xi|)>0$.
Define
\begin{equation}\label{translation}
v_\xi(x):=\varphi(|x-\xi|),\quad x\in B_R(\xi).
\end{equation}
The function $u-v_\xi$ attains a positive maximum in $B_R(\xi)$. Let $x_\xi$ be a maximum point. Then
$$
x_\xi\neq0,\quad x_\xi\neq\xi.
$$
Indeed, if $x_\xi=0$, then
$$
Du(0)=Dv_\xi(0)=-\varphi^\prime(|\xi|)\frac{\xi}{|\xi|}\neq0,
$$
contradicting $Du(0)=0$. If $x_\xi=\xi$, then $Du(\xi)=Dv_\xi(\xi)=0$, contradicting the choice of $\xi$. Thus, $x_\xi\in B_R(\xi)\setminus\{0,\xi\}$, where both functions are of class $C^2$ and $|Dv_\xi(x_\xi)|>0$. Repeating the open-and-closed argument of Case 1 in the connected domain $B_R(\xi)\setminus\{0,\xi\}$ shows that $u-v_\xi$ is constant there, contradicting its behavior on $\partial B_R(\xi)$. Therefore, a positive maximum cannot be attained at the origin. This proves \eqref{eq5.1}.
\end{proof}

\begin{remark}\label{Rem5.1}
For $q<0$, the $C^1$ requirement in the definition of a strong subsolution supplies continuity and differentiability at the origin.
In Lemma \ref{Lem5.1}, this permits the use of the gradient identities that exclude both centers after translation.
The translation \eqref{translation} is needed when $q<0$, since the strong subsolution and the radial comparison function may fail to be of class $C^2$ at the contact point and the term $|Du|^q$ is singular where the gradient vanishes. It is also used when $0<q<1$ and $f>0$, since $p\mapsto|p|^q$ is not locally Lipschitz at $p=0$.
\end{remark}
We now obtain the equivalence between entire subsolutions and global radial solutions.

\begin{theorem}[Equivalence theorem]\label{Th5.1}
Assume that $q\leq2-\frac{1}n$ and that one of the following sets of hypotheses holds.
\begin{enumerate}[label=(\alph*)]
\item Condition (A) holds and, if $q>0$, either $f(t)>0$ for all $t>0$, or $0<q<1$ and $f\equiv0$.
\item Condition (B) holds and $q>0$.
\end{enumerate}
Then the following assertions hold.
\begin{enumerate}[label=(\roman*)]
\item If $q\geq0$, equation \eqref{eq1.1} admits a positive nontrivial $(n-1)$-convex entire subsolution $u\in C^2(\mathbb R^n)$ if and only if, for some initial value $a>0$, problem \eqref{CP} admits a positive global solution $\varphi\in C^2[0,+\infty)$ satisfying $\varphi^\prime(r)>0$ and $\varphi^{\prime\prime}(r)\geq0$ for all $r>0$.

\item If $q<0$, equation \eqref{eq1.1} admits a positive nontrivial $(n-1)$-convex strong entire subsolution
$u\in W_{\mathrm{loc}}^{2,\delta}(\mathbb R^n)$ for $\delta\in [1,\frac{n(1-q)}{1-nq})$,
if and only if, for some initial value $a>0$, problem \eqref{CP} admits a positive global solution $\varphi\in C^1[0,+\infty)\cap C^2(0,+\infty)\cap W_{\mathrm{loc}}^{2,\delta}[0,+\infty)$ for $\delta\in [1,\frac{n(1-q)}{1-nq})$ satisfying
$\varphi^\prime(r)>0$ and $\varphi^{\prime\prime}(r)\geq0$ for all $r>0$.
\end{enumerate}
\end{theorem}

\begin{proof}
We first prove sufficiency. Let $q\geq0$, and let $\varphi$ be a positive global solution as in (i). Set $u(x):=\varphi(|x|)$. Lemma \ref{Lem2.1} and Remark \ref{Rem2.2} show that $u\in C^2(\mathbb R^n)$ is a positive $(n-1)$-convex entire solution of \eqref{eq1.1}. Since $\varphi^\prime>0$ on $(0,+\infty)$, the solution is nontrivial.

Let $q<0$, and let $\varphi$ be a positive global solution as in (ii). Set $u(x):=\varphi(|x|)$. Lemma \ref{Lem2.1} and Remark \ref{Rem2.2} show that $u$ is a positive nontrivial $(n-1)$-convex strong entire solution of \eqref{eq1.1} satisfying
$u\in W_{\mathrm{loc}}^{2,\delta}(\mathbb R^n)$ for $\delta\in [1,\frac{n(1-q)}{1-nq})$.
This proves sufficiency.

We next prove necessity. Let $u$ be a positive nontrivial $(n-1)$-convex entire subsolution in the class stated in (i) or (ii). In the degenerate case $0<q<1$ and $f\equiv0$, choose $x_*\in\mathbb R^n$ such that $Du(x_*)\neq0$ and replace $u(x)$ by $u(x+x_*)$. This is possible since $u$ is nontrivial, and the translated function remains a subsolution. Thus, in this case, we may assume that $Du(0)\neq0$. By the regularity required in (i)--(ii), $u(0)$ is well defined and positive in both cases. Choose
$$
a:=\frac{u(0)}{2}>0.
$$
Theorem \ref{Th3.1} under condition (A), or Theorem \ref{Th3.2} under condition (B), shows that \eqref{CP} admits a positive local solution $\varphi$ satisfying $\varphi(0)=a$, $\varphi^\prime>0$, and $\varphi^{\prime\prime}\geq0$. Choose a maximal continuation of $\varphi$ and suppose, by contradiction, that its maximal radius $R$ is finite. If $\varphi$ stayed bounded as $r\to R^-$, the upper estimate in Lemma \ref{Lemma4.2} or \ref{Lemma4.2-2} would bound $\varphi^\prime$, contradicting \eqref{finite_radius_alternative}. Hence $\lim\limits_{r\to R^-}\varphi(r)=+\infty$.
Therefore, Lemma \ref{Lem5.1} gives
$$
u(x)\leq\varphi(|x|),\quad |x|<R.
$$
Taking $x=0$, we obtain
$$
u(0)\leq\varphi(0)=a=\frac{u(0)}{2},
$$
which is impossible. Thus, problem \eqref{CP} admits a positive global increasing convex solution for some initial value $a>0$. The origin regularity follows from Lemma \ref{Lem2.2}. This proves necessity.
\end{proof}

\subsection{Proof of Theorem \ref{Th1.1}}
We combine Theorem \ref{Th5.1} with Theorem \ref{Th4.2} and Proposition \ref{Prop4.critical}.

\begin{proof}[Proof of Theorem \ref{Th1.1}]
By Theorem \ref{Th5.1}, equation \eqref{eq1.1} admits a positive nontrivial $(n-1)$-convex entire subsolution in the stated regularity class if and only if, for some $a>0$, problem \eqref{CP} admits a positive global solution satisfying $\varphi^\prime(r)>0$ and $\varphi^{\prime\prime}(r)\geq0$ for all $r>0$.

Suppose that $q<2-\frac{1}n$. By Theorem \ref{Th4.2}, such a global solution exists if and only if \eqref{KO_n} holds. If $q\geq0$, then $\varphi\in C^2[0,+\infty)$, and the associated radial function $u(x):=\varphi(|x|)$ belongs to $C^2(\mathbb R^n)$. If $q<0$, then
$\varphi\in C^1[0,+\infty)\cap C^2(0,+\infty)\cap W_{\mathrm{loc}}^{2,\delta}[0,+\infty)$ for $\delta\in [1,\frac{n(1-q)}{1-nq})$,
and Theorem \ref{Th5.1}~(ii) gives
$u\in C^1(\mathbb R^n)\cap C^2(\mathbb R^n\setminus\{0\})\cap W_{\mathrm{loc}}^{2,\delta}(\mathbb R^n)$ for
$\delta\in [1,\frac{n(1-q)}{1-nq})$.
This proves (i) and (ii).

Let $q=2-\frac{1}n$. If $g\equiv0$ on $(0,+\infty)$, Proposition \ref{Prop4.critical}~(i) and Theorem \ref{Th5.1}~(i) give the equivalence with \eqref{KO_n_critical_zero}. If $g\not\equiv0$ on $(0,+\infty)$, Proposition \ref{Prop4.critical}~(ii) shows that \eqref{CP} admits no positive global solution. By Theorem \ref{Th5.1}~(i), \eqref{eq1.1} admits no positive nontrivial entire subsolution. This proves (iii).
\end{proof}

\subsection{Proof of Theorem \ref{Th1.2}}

Under condition (B), the sign of $g$ at infinity determines which Keller--Osserman condition governs global solvability.

\begin{proof}[Proof of Theorem \ref{Th1.2}]
Since $q>0$ and condition (B) gives $f(t)>0$ for $t>0$, positive classical subsolutions are nonconstant. Theorem \ref{Th5.1}~(i) shows that equation \eqref{eq1.1} admits a positive $(n-1)$-convex entire subsolution $u\in C^2(\mathbb R^n)$ if and only if, for some $a>0$, problem \eqref{CP} admits a positive global solution satisfying $\varphi^\prime(r)>0$ and $\varphi^{\prime\prime}(r)\geq0$ for all $r>0$.

Let $0<q<2-\frac{1}n$ and $\lim\limits_{t\to+\infty}g(t)>0$. By Theorem \ref{Th4.2-2}~(i), such a global solution exists if and only if \eqref{KO_n_2} holds. Hence, \eqref{KO_n_2} is necessary and sufficient for the existence of a positive $(n-1)$-convex entire subsolution.

Let $0<q\leq2-\frac{1}n$ and $\lim\limits_{t\to+\infty}g(t)\leq0$. By Theorem \ref{Th4.2-2}~(ii), a positive global solution exists if and only if \eqref{KO_n_3} holds. Together with Theorem \ref{Th5.1}, this proves (ii).

Finally, let $q=2-\frac{1}n$ and $\lim\limits_{t\to+\infty}g(t)>0$. Proposition \ref{Prop4.critical}~(ii) shows that \eqref{CP} admits no positive global solution. By Theorem \ref{Th5.1}~(i), equation \eqref{eq1.1} admits no positive $(n-1)$-convex entire subsolution. This proves (iii).
\end{proof}

\begin{remark}\label{Rem5.3}
The positivity requirement on $u$ and $\varphi$ can be removed if $f$ is strictly positive on $\mathbb R$. More precisely, replace conditions (A) and (B) by
\begin{itemize}
\item[$(A)'$] $f\in C(\mathbb R,(0,+\infty))$ and $g\in C(\mathbb R,[0,+\infty))$ are non-decreasing;
\item[$(B)'$] $f\in C(\mathbb R,(0,+\infty))$ and $g\in C(\mathbb R,\mathbb R)$ are non-decreasing.
\end{itemize}
Under these alternative conditions, Lemma \ref{Lem5.1} and Theorem \ref{Th5.1} remain valid without requiring $u$ or $\varphi$ to be positive, and the initial value in \eqref{CP} may be any $a\in\mathbb R$. In the necessity proof, one chooses any $a<u(0)$ in place of $a=\frac{u(0)}{2}$. The Keller--Osserman conditions are unchanged, since changing a finite lower integration limit does not affect convergence at infinity.
\end{remark}


\section{Examples and Remarks}\label{Section6}

The following examples illustrate the effects of coefficient growth and the sign of the gradient term in Theorems \ref{Th1.1} and \ref{Th1.2}. Throughout this section, an entire subsolution is understood to be nontrivial and $(n-1)$-convex. For $q\geq0$, it belongs to $C^2(\mathbb R^n)$, and for $q<0$, it is understood in the strong class $u\in W_{\mathrm{loc}}^{2,\delta}(\mathbb R^n)$ for $\delta\in[1,\frac{n(1-q)}{1-nq})$.

\subsection{Examples under Condition (A)}

\begin{example}[Power nonlinearities]\label{Ex6.power}
Let $p,m\geq0$ and $q\leq 2-\frac{1}n$. Consider
\begin{equation*}
\operatorname{det}^{\frac{1}n}(\Delta uI-D^2u)=u^p+u^m|Du|^q\quad\text{in }\mathbb R^n.
\end{equation*}
If $q<2-\frac{1}n$, an entire subsolution exists if and only if $p\leq1$ and $m\leq1-q$. If $q=2-\frac{1}n$, since $g(t)=t^m>0$ for $t>0$, no such subsolution exists.
\end{example}

\begin{example}[Pure-gradient nonlinearity]\label{Ex6.degenerate}
Let $m\geq0$ and $q<1$. Consider
\begin{equation*}
\operatorname{det}^{\frac{1}n}(\Delta uI-D^2u)=u^m|Du|^q\quad\text{in }\mathbb R^n.
\end{equation*}
A positive entire subsolution exists if and only if $m\leq1-q$. When $0<q<1$, this condition concerns immediately increasing radial solutions, and positive constants also solve the equation, independently of the condition.
\end{example}

\begin{example}[Power-logarithmic nonlinearities]\label{Ex6.powerlog}
Let $p,m\geq0$ and $q\leq 2-\frac{1}n$. Consider
\begin{equation*}
\operatorname{det}^{\frac{1}n}(\Delta uI-D^2u)=u^m(\ln(1+u))^p(1+|Du|^q)\quad\text{in }\mathbb R^n.
\end{equation*}
If $q<2-\frac{1}n$, a positive  entire subsolution exists if and only if
\begin{equation*}
\begin{cases}
m<1,\quad\text{or}\quad m=1\ \text{and}\ p\leq 2-\frac{1}n,&q\leq0,\\
m<1-q,\quad\text{or}\quad m=1-q\ \text{and}\ p\leq 2-\frac{1}n-q,&0<q<2-\frac{1}n.
\end{cases}
\end{equation*}
In particular, when $q=1$, the second condition reduces to $m=0$ and $p\leq1-\frac{1}n$. When $q=2-\frac{1}n$, no such subsolution exists.
\end{example}

\begin{example}[Exponential nonlinearities]\label{Ex6.exp}
Let $\alpha,\beta\geq0$ and $q\leq 2-\frac{1}n$. Consider
\begin{equation*}
\operatorname{det}^{\frac{1}n}(\Delta uI-D^2u)=e^{\alpha u}+e^{\beta u}|Du|^q\quad\text{in }\mathbb R^n.
\end{equation*}
If $q<2-\frac{1}n$, an entire subsolution exists if and only if $\alpha=\beta=0$ and $q\leq1$. When $q=2-\frac{1}n$, no such subsolution exists.
\end{example}

\subsection{Examples under Condition (B)}

In this subsection, $0<q\leq 2-\frac{1}n$.

\begin{example}[Sign-changing power coefficient]\label{Ex6.sign}
Let $p\geq0$, $m>0$, and $c>0$. Consider
\begin{equation*}
\operatorname{det}^{\frac{1}n}(\Delta uI-D^2u)=1+u^p+(u^m-c)|Du|^q\quad\text{in }\mathbb R^n.
\end{equation*}
If $0<q<2-\frac{1}n$, a positive entire subsolution exists if and only if $p\leq1$ and $m\leq1-q$. When $q=2-\frac{1}n$, since $g(t)=t^m-c$ is positive for all sufficiently large $t$, no such subsolution exists.
\end{example}

\begin{example}[Decaying negative power coefficient]\label{Ex6.negative}
Let $p,m\geq0$. Consider
\begin{equation*}
\operatorname{det}^{\frac{1}n}(\Delta uI-D^2u)=1+u^p-(1+u)^{-m}|Du|^q\quad\text{in }\mathbb R^n.
\end{equation*}
A positive entire subsolution exists if and only if $p\leq1$ or $p+m\leq q$. In particular, if $m=0$ and $1<p\leq q\leq2-\frac1n$, the negative gradient term permits a positive entire subsolution, although the equation with the same zero-order term and $g\equiv0$ admits none. This gives a direct example of the compensation measured by \eqref{KO_n_3}.
\end{example}

\begin{example}[Decaying negative exponential coefficients]\label{Ex6.negative_exp}
Let $\beta>0$. The following assertions hold for $0<q\leq 2-\frac{1}n$.
\begin{enumerate}[label=(\roman*)]
\item If $\alpha\geq0$, the equation
\begin{equation*}
\operatorname{det}^{\frac{1}n}(\Delta uI-D^2u)=e^{\alpha u}-e^{-\beta u}|Du|^q\quad\text{in }\mathbb R^n
\end{equation*}
admits an entire subsolution if and only if $\alpha=0$.

\item If $p\geq0$, the equation
\begin{equation*}
\operatorname{det}^{\frac{1}n}(\Delta uI-D^2u)=(1+\ln(1+u))^p-e^{-\beta u}|Du|^q\quad\text{in }\mathbb R^n
\end{equation*}
admits a positive entire subsolution for any $p\geq0$.
\end{enumerate}
\end{example}

\subsection{Further Remarks}

\begin{remark}\label{Rem6.previous_work}
The equation in \cite{JJ2026} is $\det(\Delta uI-D^2u)=\psi(x,u,Du)$, where we denote its inhomogeneous term by $\psi$. For $(n-1)$-convex solutions of \eqref{eq1.1}, the corresponding expression is
$$
\psi(x,z,\xi)=\bigl(f(z)+g(z)|\xi|^q\bigr)^n.
$$
For $q>0$ and $g(z)>0$, its leading gradient power is $nq$. Thus the power $n$ in condition $(G_2)$ of \cite{JJ2026} corresponds to $q=1$ in the present normalization. That condition is an upper growth assumption used together with other structural hypotheses to establish boundary blow-up solutions on bounded domains. For the present additive equation, our results give nonexistence of positive entire subsolutions and existence of strictly increasing convex radial solutions with finite-radius blow-up when $1<q\leq2-\frac1n$ and $g$ is eventually positive, under the stated hypotheses. For $0<q\leq1$, global existence remains determined by the corresponding Keller--Osserman conditions.

Under condition $(G_1)^-$, recalled in Table \ref{Tab6.previous_work}, Lemma 5.3 of \cite{JJ2026} constructs an entire subsolution. This construction is used in the proof of Theorem 1.4 there to establish the nonexistence of boundary blow-up solutions on bounded domains.
For the exact model $\psi(x,z,\xi)=(1+(z^+)^\mu)(1+|\xi|^\nu)$ with $\mu\geq0$ and $0\leq\nu\leq2n-1$, the present results also establish the necessity of this exponent restriction. Indeed, taking $f(t)=g(t)=(1+(t^+)^\mu)^{\frac1n}$ and $q=\frac\nu n$, we have
$$
2^{\frac1n-1}\bigl(f(z)+g(z)|\xi|^q\bigr)\leq\psi(x,z,\xi)^{\frac1n}\leq f(z)+g(z)|\xi|^q.
$$
Comparison with these additive inhomogeneous terms and Theorem \ref{Th1.1} show that the model equation admits a positive nontrivial $(n-1)$-convex entire subsolution if and only if $\mu+\nu\leq n$. Thus the range covered by $(G_1)^-$ is recovered, and nonexistence of entire subsolutions is established when $\mu+\nu>n$. This necessity applies to the specified model and does not follow merely from the failure of a one-sided upper bound for a general $\psi$. Table \ref{Tab6.previous_work} summarizes these relations.
\end{remark}

\begin{table}[htbp]
	\caption{Comparison of assumptions and conclusions in the preceding papers and Part I.}\label{Tab6.previous_work}
	\small
	\renewcommand{\arraystretch}{1.15}
	\begin{tabularx}{\textwidth}{@{}>{\raggedright\arraybackslash}p{0.15\textwidth}>{\raggedright\arraybackslash}X>{\raggedright\arraybackslash}X@{}}
		\toprule
		Reference and condition & Equation and growth assumptions & Main conclusions and relation to Part I \\
		\midrule
		Jiang-Ji-Li \cite{JJL2025}, Theorem 1.1 & $\det^{\frac1n}(\Delta uI-D^2u)=b(x)f(u)$ with $b\equiv1$; condition \eqref{KO_n_critical_zero}. & A necessary and sufficient condition for the existence of entire subsolutions, recovered in Part I when $g\equiv0$. \\
		\addlinespace
		Ji-Jiang \cite{JJ2026}, $(G_2)$ & $\psi(x,z,\xi)\leq G(z)(1+|\xi|^n)$, where $G$ is positive and smooth and satisfies $\sup_{z\leq0}e^{-\varepsilon z}G(z)<+\infty$ for some $\varepsilon>0$. & An upper growth bound forming part of hypotheses $(G_1)$--$(G_4)$ for the existence of boundary blow-up solutions on bounded domains. The gradient exponent $n$ corresponds to $q=1$ in the present normalization. \\
		\addlinespace
		Ji-Jiang \cite{JJ2026}, $(G_1)^-$ & $0\leq\psi(x,z,\xi)\leq M(1+(z^+)^\mu)(1+|\xi|^\nu)$, where $\mu,\nu\geq0$ and $\mu+\nu\leq n$. & Lemma 5.3 establishes the existence of an entire subsolution, and Theorem 1.4 proves nonexistence of boundary blow-up solutions on bounded domains. For the exact product model with $0\leq\nu\leq2n-1$, Part I additionally establishes the necessity of $\mu+\nu\leq n$ for the existence of entire subsolutions. \\
		\addlinespace
		Part I, conditions (A) and (B) & $\det^{\frac1n}(\Delta uI-D^2u)=f(u)+g(u)|Du|^q$; conditions \eqref{KO_n}--\eqref{KO_n_3}. & Necessary and sufficient conditions for the existence of entire subsolutions for $q<2-\frac1n$ under (A) and $0<q<2-\frac1n$ under (B), together with existence and nonexistence results for $q=2-\frac1n$. The results also include the existence of radial solutions with finite-radius blow-up when the corresponding global existence condition fails.\\
		\bottomrule
	\end{tabularx}
\end{table}

\begin{remark}\label{Rem6.critical}
Within the hypotheses of Theorems \ref{Th1.1} and \ref{Th1.2}, a gradient coefficient that is positive at infinity prevents the existence of an entire subsolution when $1<q\leq 2-\frac{1}n$. At the critical gradient exponent $q=2-\frac{1}n$, condition (A) permits an entire subsolution only when $g\equiv0$ on $(0,+\infty)$ and \eqref{KO_n_critical_zero} holds. Under condition (B), a nonpositive gradient coefficient can permit entire subsolutions even at $q=2-\frac{1}n$, as Example \ref{Ex6.negative} shows.
\end{remark}

\begin{remark}\label{Rem6.origin}
When $0<q<1$ and $f\equiv0$, constant solutions coexist with the immediately increasing branch. This explains the nontriviality requirement in Theorem \ref{Th1.1}. For $q<0$, the common regularity class is the strong class stated at the beginning of this section. In the singular examples under condition (A), the displayed gradient coefficients satisfy $g(a)>0$ for every $a>0$. For the corresponding constructed radial solutions, \eqref{O} therefore gives $u\in W_{\mathrm{loc}}^{2,\delta}(\mathbb R^n)$ for $1\leq\delta<n(1-\frac{1}q)$.
\end{remark}

\noindent{\bf Acknowledgements.}

J. Ji was supported by the SEU Innovation Capability Enhancement Plan for Doctoral Students (Grant No. CXJH\_SEU 24123).
F. Jiang was supported by the National Natural Science Foundation of China (Grant No. 12271093), and by the Shanghai Institute for Mathematics and Interdisciplinary Sciences (Grant No. SIMIS‑ID‑2025‑AD).
M. Li was supported by the Alexander von Humboldt Research Fellowship.
F. Jiang and M. Li were supported by the Jiangsu Provincial Scientific Research Center of Applied Mathematics (Grant No. BK20233002).

\bibliographystyle{amsplain}

\end{document}